\author{\Large{Damanvir Singh Binner and Amarpreet Rattan}}
\documentclass[12pt]{article}
\usepackage[margin=1in]{geometry}
\usepackage{amsfonts,amsmath,amssymb, amsthm, tikz}
\usepackage{xcolor,colortbl}
\usepackage[notcite, notref,color,final]{showkeys}
\usepackage[flushleft]{threeparttable}
\usepackage{hyperref}
\definecolor{refkey}{rgb}{0,0,1}
\definecolor{labelkey}{rgb}{0,0,1}
\usepackage{fancyvrb}

\usepackage{lineno}
\newtheorem{theorem}{Theorem}
\newtheorem{conjecture}[theorem]{Conjecture}
\newtheorem{lemma}[theorem]{Lemma}
\newtheorem{corollary}[theorem]{Corollary}

\theoremstyle{remark}
\newtheorem*{remark}{Remark}

\newcommand{\pihat}{\hat{\pi}}

\begin{document}
\title{\Large{On Conjectures Concerning the Smallest Part and Missing Parts of Integer Partitions}}

\date{}
\maketitle
\newcommand{\cmagenta}{\color{black}}
\newcommand{\cred}{\color{black}}
\newcommand{\cskyblue}{\color{skyblue}}
\newcommand{\cnavyblue}{\color{navyblue}}
\newcommand{\cpurple}{\color{black}}
\newcommand{\cgreen}{\color{green}}
\newcommand{\cbrown}{\color{brown}}
\newcommand{\corange}{\color{black}}
\newcommand{\cyellow}{\color{yellow}}
\newcommand{\freq}{f}

\begin{center}
\vspace*{-8mm}
\large{Department of Mathematics\\
 Simon Fraser University\\
 Burnaby, BC V5A 1S6 \\
 Canada\\
dbinner@sfu.ca\\
rattan@sfu.ca}
\end{center}

\begin{center}
\large{Abstract}
\end{center} 
For positive integers $L \geq 3$ and $s$,  Berkovich and Uncu (Ann. Comb. $23$ ($2019$) $263$--$284$)
 conjectured an inequality between the sizes of two closely related sets of partitions whose parts lie in the interval $\{s,
\ldots, L+s\}$.  Further restrictions are placed on the sets by
specifying impermissible parts as well as a minimum part.  The authors proved their conjecture for the
cases $s=1$ and $s=2$.  In the present article, we prove their conjecture for
general $s$ by proving a stronger theorem.  We also prove other related conjectures found in the same paper.

\section {Introduction}
\label{Intro}

Let $n$ be a nonnegative integer. A \emph{partition $\pi = (\pi_1,
\pi_2, \dots)$ of $n$} is a weakly decreasing list of positive
integers whose sum is $n$, and we write $|\pi| = n$ to indicate this.  We
allow the empty partition as the unique partition of $0$.  Each $\pi_i$
is known as a \emph{part} of $\pi$.  In the present article, it is
more convenient to use the notation that expresses the number of parts
of each size in a partition.   In this notation, we write $\pi = (1^{f_1}, 2^{f_2},
\ldots)$, where $f_i$ is the \emph{frequency} of $i$ or the number of times a
part $i$ occurs in $\pi$.   Thus, each frequency $f_i$ is a nonnegative integer, and when
$f_i = 0$ this expresses that $\pi$ has no part of size $i$.  When the frequency
of a number is 0, it may or may not be omitted in the expression.  In the latter
notation, it is clear that $|\pi| = \sum_i i \cdot f_i$.   Thus $(4,4,2,2,1)$,
$(1^1, 2^2, 3^0, 4^2, 6^0)$ and $(1^1, 2^2, 4^2, 5^0)$ all represent the same
partition of 13.

In \cite{BerkovichAlexander2017SEPI}, Berkovich and Uncu
conjectured some intriguing partition inequalities regarding the relative sizes of certain
sets of partitions.  We recall their definitions. For positive integers $L$ and $s$, 
\begin{itemize}
	\item $C_{L,s}$ denotes the set of partitions where the smallest
		part is $s$, all parts are $\leq L+s$, and $L+s-1$ does not
		appear as a part; 
	\item $D_{L,s}$  denotes the set of nonempty partitions
		with parts in the set $\{s+1,  \ldots,
		L+s\}$. 
	\end{itemize}
Conjecture \ref{Main} is found in \cite[Conjecture
3.2]{BerkovichAlexander2017SEPI}, and one focus of this article
is to prove it.
\label{page1}
\begin{conjecture}[Berkovich and Uncu (2019)]
\label{Main}
For positive integers $L \geq 3$ and $s$, there exists an $M$,
which only depends on $s$, such that $$ |\{\pi \in C_{L,s} : |\pi| =
N \}| \geq |\{\pi \in D_{L,s} : |\pi| = N\}|,$$ for every $N \geq
M$.
\end{conjecture}
They proved in \cite[Theorem 1.1, Theorem 3.1]{BerkovichAlexander2017SEPI}
Conjecture \ref{Main} for $s=1$ (with bound $M=1$) and $s=2$ (with bound
$M=10$).     In both cases, the authors found a suitable injection.  Conjecture
\ref{Main} is therefore a natural generalization of those theorems.  Their
investigations suggested further conjectures, three of which we give below.  To
state the first, for positive integers $L$ and $s$, if $L \geq s+1$, 
\begin{itemize}
\item  $C^{*}_{L,s}$ denotes the set of
		partitions where the smallest part is $s$, all
		parts are $\leq L+s$, and $L$ does not appear as a
		part.
\end{itemize}
The next conjecture is found in \cite[Conjecture
3.3]{BerkovichAlexander2017SEPI}.
\begin{conjecture}[Berkovich and Uncu (2019)] \footnote{{\cmagenta For recent progress on
		both Conjectures \ref{Main} and \ref{Similar}, see Section
\ref{sec:recent}.}}
    \label{Similar}
    For positive integers $L \geq 3$ and $s$, there exists an $M$,
    which only depends on $s$, such that 
    \begin{equation*}
        |\{\pi \in C^{*}_{L,s} : |\pi| = N \}| \geq |\{\pi \in D_{L,s} :
        |\pi| = N\}|,
    \end{equation*}
	for every $N \geq M$.
\end{conjecture}
In the definition of $C^{*}_{L,s}$, we must have $L \geq s+1$, so the inclusion
of $L \geq 3$ in the conjecture is to exclude the case $L = 2$ and $s=1$. 

Conjectures \ref{Main} and \ref{Similar} are part of a broader body of
recent work concerning sets of partitions whose parts come from some interval.  See for
example \cite{abr, berkbook, chapman}.   While we further resolve additional
related conjectures from \cite{BerkovichAlexander2017SEPI} below, there are a number
of other research directions suggested in that article that we do not pursue here.  

While Conjectures \ref{Main} and \ref{Similar} motivated our work, we in fact
prove a stronger result.  For positive integers $L$, $s$ and $k$, with $s+1 \leq k \leq L+s$,
\begin{itemize}
	\item {\corange $I_{L,s,k}$} is the set of
partitions where the smallest part is $s$, all parts are $\leq L+s$, and $k$ does
not appear as a part.
\end{itemize}
Whenever a part cannot occur from a range of allowable parts, as with $k$ in the
definition of {\corange $I_{L,s,k}$}, we refer to that as an \emph{impermissible part}.  
The sets $C_{L,s}$ and $C^{*}_{L,s}$ above are the special cases of
{\corange $I_{L,s,k}$} given by
{\corange $I_{L,s,L+s-1}$} and {\corange $I_{L,s,L}$}, respectively.  Thus, the parameter $k$ allows us to deal
with impermissible parts in the set $\{s+1, \ldots, L+s\}$ collectively. The
next theorem generalizes Conjectures \ref{Main} and \ref{Similar}.   
\begin{theorem}
\label{MostGen}
For positive integers $L$, $s$ and $k$, with $L \geq 3$ and $s+1 \leq k \leq
L+s$, we have    
\begin{equation*}
	|\{\pi \in {\corange  I_{L,s,k}} : |\pi| = N \}| > |\{\pi \in D_{L,s} : |\pi| =
        N\}|,
    \end{equation*}
for all $N \geq \Gamma(s)$, where $\Gamma(s)$ is defined in \eqref{eq:defgammas}.
\end{theorem}
At this point, the precise value of $\Gamma(s)$ is not important.  We have, however,
stated Theorem \ref{MostGen} with the constant $\Gamma(s)$ inserted to
emphasize that it is explicitly known and only depends on $s$.  It also allows
us to easily reference this bound when using the partition inequality presented in Theorem
\ref{MostGen} to prove other results.  We prove Theorem \ref{MostGen} in Section \ref{sec:sim}. While our methods are elementary and involve constructing injective maps between the
relevant sets, they entail analyzing many cases.

For the remaining conjectures of Berkovich and Uncu considered in the present
article, define
\begin{itemize}
    \item the \emph{$q$-Pochhammer symbol} by $$(a;q)_n := (1-a)(1-aq)(1-aq^2)
        \cdots (1-aq^{n-1}),$$
		for an integer $n \geq 1$, with {\cmagenta $(a;q)_0 := 1$};
\item the series $H_{L,s,k}(q)$ by
    \begin{equation*}
        H_{L,s,k}(q) := \frac{q^s(1-q^k)}{(q^s;q)_{L+1}} -
        \left(\frac{1}{(q^{s+1};q)_L}-1\right),
    \end{equation*}
    for positive integers $L,s$ and $k$. 
\end{itemize}
A series $\sum_{n \geq 0} a_n q^n$ is said to be \emph{eventually
positive} if there exists some $l \in \mathbb{N}$ such that $a_n > 0$ for all $n \geq l$.  
The next conjecture is found in \cite[Conjecture 7.1]{BerkovichAlexander2017SEPI}.
\begin{conjecture}[Berkovich and Uncu (2019)] \footnote{{\cmagenta See Footnote 1, Page
	\pageref{page1}}}
    \label{Analytic}
    For positive integers $L$, $s$ and $k$, with $L \geq 3$ and $k \geq s+1$, the series $H_{L,s,k}(q)$ is eventually positive.
\end{conjecture}
As stated, the bound $l$ guaranteeing the coefficient of $q^N$ in $H_{L,s,k}(q)$
is positive for all $N \geq l$ may depend on $L$, $s$ or $k$ in Conjecture \ref{Analytic}.   When $s+1 \leq k \leq
L+s$, elementary partition theory gives the coefficient of $q^N$ in $H_{L,s,k}(q)$ 
as 
\begin{equation}\label{eq:coeffHdef}
	|\{\pi \in {\corange  I_{L,s,k}} : |\pi| = N \}| - |\{\pi \in D_{L,s} : |\pi| =
        N\}|.  
\end{equation}
Hence, Theorem \ref{MostGen} proves
Conjecture \ref{Analytic} when $s+1 \leq k \leq L+s$, and indeed Conjectures
\ref{Main}, \ref{Similar} and \ref{Analytic} motivated Theorem \ref{MostGen}.  However, Conjecture \ref{Analytic} is
valid for values of $k$ that do not have the combinatorial interpretation
specified in \eqref{eq:coeffHdef}.  We prove a result stronger than
Conjecture \ref{Analytic} that also generalizes Theorem \ref{MostGen}.
\begin{theorem}
\label{Genk}
  For positive integers $L$, $s$ and $k$, with $L \geq 3$ and $k \geq s+1$, the coefficient of
  $q^N$ in $H_{L,s,k}(q)$ is positive whenever $N \geq \Gamma(s)$, where
  $\Gamma(s)$ is defined in \eqref{eq:defgammas}.
\end{theorem}
Again, we emphasize that the bound given in Theorem \ref{Genk} only depends on
$s$, is explicitly known, and is the same as the bound in Theorem
\ref{MostGen}.  We use Theorem \ref{MostGen} along with other results to prove Theorem
\ref{Genk} in Section \ref{sec:bigk}.

For the last conjecture of Berkovich and Uncu considered here, given a positive integer $L$,
\begin{itemize}
\item $G_{L,2}(q)$ is the series
    \begin{equation*}
        G_{L,2}(q): = \sum_{\substack{\pi \in \mho, \\ s(\pi)=2, \\ l(\pi)-s(\pi)
        \leq L}} q^{|\pi|} -  \sum_{\substack{\pi \in \mho, \\ s(\pi) \geq 3, \\
        l(\pi)-s(\pi) \leq L}} q^{|\pi|}, 
    \end{equation*}
    where $s(\pi)$ and $l(\pi)$ denote the smallest and largest parts of $\pi$,
    respectively, and $\mho$ denotes the set of partitions $\pi$ with $|\pi| >
    0$.
\end{itemize}
A series $ S = \sum_{n \geq
0} a_nq^n $ is said to be \emph{nonnegative} if $a_n \geq 0$ for all $n$. The
nonnegativity of the series $S$ is denoted by $ S \succeq 0. $  We prove the
next conjecture, found in \cite[Conjecture 5.3]{BerkovichAlexander2017SEPI}, in Section \ref{GL2}.
\begin{conjecture}[Berkovich and Uncu (2019)]
\label{s=2}
For $L=3$, $$G_{L,2}(q)+q^3 + q^9 + q^{15} \succeq 0; $$ for $L=4$, $$G_{L,2}(q)+q^3 + q^9 \succeq 0; $$ and for $L \geq 5$, $$ G_{L,2}(q)+q^3  \succeq 0.$$
\end{conjecture}
\begin{remark}
	Our statement of Conjecture \ref{s=2} differs slightly from the one given in 
	\cite{BerkovichAlexander2017SEPI}, as their statement is not strictly correct.  In the
	case $L=3$, the conjecture in \cite{BerkovichAlexander2017SEPI} is stated as $G_{3,2}(q) + q^3 + q^9 \succeq 0$.  However, it
    can be checked, either through machine computation or by hand, that the coefficient of
	$q^{15}$ in $G_{3,2}(q)$ is -1,  and hence the discrepancy between our statement and
	theirs.  Subject to this minor modification, their conjecture is as
    stated above.  We further remark that in \cite{BerkovichAlexander2017SEPI}
    the authors prove a similar nonnegativity result for a series
    $G_{L,1}(q)$.  The definition of $G_{L,1}(q)$ is the same as the one for
    $G_{L,2}(q)$, except the restriction on $s(\pi)$ is that $s(\pi) = 1$ in the
    first sum and $s(\pi) \geq 2$ in the second. 
\end{remark}

Our proofs rely heavily on the following two lemmas, the first of which is a well-known result of Sylvester.  
\begin{lemma}[Sylvester (1882)]
\label{Frobenius}
For natural numbers $a$ and $b$ such that $\gcd(a,b)=1$, the equation
$ax+by=n$ has a solution $(x, y)$, with $x$ and $y$ nonnegative
integers, whenever $n \geq (a-1)(b-1)$. 
\end{lemma} 

The quantity $(a-1)(b-1)$ is known as the 
\emph{Frobenius number}\footnote{The Frobenius number for two coprime integers
	$a$ and $b$
	is usually defined as the largest integer \emph{not} expressible as a nonnegative
	linear combination of $a$ and $b$, and this number is always $(a-1)(b-1) - 1$.   Here,
	we have chosen to call the smallest number $N$ such that every $n \geq N$
	can be written as a linear combination of $a$ and $b$ the Frobenius number.
As indicated, it is known that $N=(a-1)(b-1)$.} of $a$ and $b$, and it is known
to be the smallest number with the property in Lemma \ref{Frobenius}.  Lemma \ref{Frobenius} was first known to Sylvester in
\cite{Sylvester82}, but we refer the reader to the four proofs in \cite[Pages
31-34]{alfonsin}.   More recently, an elementary proof was given in
\cite[Corollary 12]{Binner:2019aa}.  In addition to Sylvester's theorem, we also require the following simple lemma.

\begin{lemma}
\label{Simple}   
Let $s$ and $n$ be positive integers such that $n \geq s+1$.  Then the
equation $$ (s+1)X_{s+1} + (s+2)X_{s+2} + \cdots + (2s+1)X_{2s+1} = n
$$ has a solution $(X_{s+1}, X_{s+2}, \dots, X_{2s+1})$, where $X_i$ is a nonnegative integer for all $i$.
\end{lemma}

\begin{proof}
We use the division algorithm to write $n=(s+1)q+r$ for some $q \geq
1$ and $0 \leq r \leq s$. If $r=0$, setting $X_{s+1} = q$ with all other $X_i
=0$ gives a suitable solution.  Otherwise $1 \leq r \leq s$, and then $$ n = (s+1)(q-1) + (s+1+r).$$ Note
that $ s+2 \leq s+1+r \leq 2s+1 $, so setting $X_{s+1} = q-1$ and
$X_{s+1+r} = 1$ with all other $X_i = 0$ gives a solution, completing the proof.
\end{proof}

\subsection{Recent Proofs of the Above Conjectures}\label{sec:recent}

Recently, Zang and Zeng \cite{zang} gave proofs of
Conjectures \ref{Main}, \ref{Similar} and \ref{Analytic}.  We
highlight the specific similarities and differences between their proofs and
ours, as well as the strengths of both approaches.
In our case, our approach almost always involves constructing an
injective map between the relevant sets of partitions, and those maps heavily rely on Lemmas \ref{Frobenius} and \ref{Simple}.   Furthermore, our
approach allows us to give, in all cases, explicit bounds for when the partition inequalities hold.  Our methods also allow us 
to prove Conjecture \ref{s=2}, which was unresolved until now.

In their paper and ours, a crucial step is to prove Conjecture \ref{Analytic} and then use it
to prove the other conjectures.  Another similarity is that the proofs of
Conjecture \ref{Analytic} are separated into two cases:  the case with
large $L$ and $k$ (compared to $s$) and the case with small $L$ or $k$.  The
comparisons between the two approaches to Conjecture \ref{Analytic} are as follows.
\begin{itemize}
\item An important achievement in both papers is to show that there exists an $M$, which
	depends only on $s$, such that the coefficient of $q^N$ in $H_{L,s,k}(q)$ is positive whenever $N \geq M$.  Thus, our work and theirs
    both achieve this strengthening of Conjecture \ref{Analytic}.  They prove
    this strengthening only for $\max(s+1,L) \leq k \leq L+s$ (\cite[Theorem
    1.1]{zang}), while we show this for any $k \geq s+1$ (Theorem \ref{Genk}).
    \item For large $L$ and $k$, their proofs and ours are different,
        as are the lower bounds on $L$ and $k$ for when these results hold.  The bounds on $N$ guaranteeing
        the coefficient of $q^N$ in $H_{L,s,k}(q)$ is positive are similar.  The
        lower bounds on $L$ and $k$ are lower in our case ($L \geq s+2$ and $k
        \geq 2s+2$) versus their case ($k \geq L \geq 2s^3 + 5s^2 +1$).  For
        their results and ours, the coefficient of $q^N$ in the series $H_{L, s,
        k}(q)$ is positive when $N$ exceeds a lower bound of order  $O(s^5)$.
       However, their restriction on $k$, that $k \geq L$, means that for a given $s$, if $L$
       becomes arbitrarily large, their result does not have this lower bound on
       $N$ for an arbitrarily large set of values of $k$, whereas in our case
       the bound is valid whenever $k \geq 2s+2$.  This is especially important
       in the limiting case $L \rightarrow \infty$. 
       For example, for a given $s$, for any fixed $k \geq 2s+2$, and any $N >
       (25s)^{5}$,  our approach shows that the number of partitions of $N$ with smallest part $s$ and no part equal to $k$ is more than the number of partitions of $N$ with smallest part greater than $s$.  
        On the other hand, their approach does not yield any such result since
        they require $k \geq L$ and here $L \rightarrow \infty$.

            \item For small $L$ or $k$,  their approach and ours differ greatly.  In \cite{zang}, the authors use 
        a celebrated result of Frobenius and Schur (related to the
        \emph{Frobenius coin problem}), which states that for a set $A = \{a_1, \ldots,
        a_m\}$ of positive integers whose greatest common divisor is 1, the number of
        partitions of a positive integer $n$ whose parts are restricted to $A$
        is approximately
        \begin{equation*}
            \frac{n^{m-1}}{(m-1)! a_1 \cdots a_m}.
        \end{equation*}
        This result is asymptotic in $n$, and it is unknown when this
        approximation is accurate.   In fact, even finding for which $n$ onwards there is
	at least one such partition is a well-known open problem;  see \cite{alfonsin}.
	Therefore, the result of Zang and Zeng is also asymptotic.
        In contrast, our methods are combinatorial, and we produce explicit
        bounds on when $H_{L,s,k}(q)$ is eventually positive.
        For large $L$ and small $k$ ($L \geq 3s+3$ and $s+1 \leq k \leq 2s+1$),
        our bounds are $O(s^{10})$, while for small $L$ ($L \leq 3s+2$), our          
        bounds are of the order $O\Big((6s)^{(6s)^{18s}}\Big)$. 

        Some advantages of the proof of Zang and Zeng in this case is that it is
        short, elegant, and easily understood.  Also, their methods show an
        intriguing connection between the present problems
    and the above mentioned theorem of Frobenius and Schur.   They further show, in \cite[Theorem
    1.3]{zang}, the eventual positivity of a series
    $H^*_{L,s,r,k_1,k_2}(q)$ that generalizes the series $H_{L,s,k}(q)$. In this case also, their results are asymptotic.  The chief
        advantage of our methods is that they produce explicit bounds, and they 
        also lead to a proof of Conjecture
        \ref{s=2}.  Indeed, in \cite[Page 12]{zang}, the authors state that
        techniques that produce explicit bounds on when $H_{L,s,k}(q)$ is
        eventually positive may lead to a proof of Conjecture \ref{s=2}.  Our methods confirm this.
\end{itemize}
A final remark about our results is that while we find explicit bounds
throughout this paper, we make no claims about the optimality of those bounds.  The question of finding the minimal bounds for
        when these results hold remains open.

\section{The case when $L$ is large for Theorem \ref{MostGen}}
\label{Proofs}

In this section, we build to proving Theorem \ref{MostGen} when $L$ is
relatively large, by which we mean $L$ is larger than a constant times $s$. In each
case, our lower bound on $L$ is explicitly stated.  We begin with a case
pertaining to Conjecture \ref{Main}, and then later generalize those arguments to prove
the large $L$ case of Theorem
\ref{MostGen}.  The general technique used in this manuscript will be
illustrated in this section.

\subsection{The case when the impermissible part $k$ is $L+s-1$, and $L$ is
relatively large}

In this section, we focus on the case $k=L+s-1$ in Theorem \ref{MostGen} with 
$L \geq s+3$.  That is, we are considering the case $C_{L,s}$ (or equivalently
$I_{L,s,L+s-1}$) when $L \geq s+3$.  This corresponds to the case of $L \geq
s+3$ in Conjecture \ref{Main}.

For any $s \geq 1$, define the quantities:

\begin{itemize}
\item $F(s) = (10s-2)(15s-3)+8s$; 
\item $\kappa(s) = (12s-1)((s+1)+(s+2)+\cdots (F(s)-1))+1$.
\end{itemize}

The number $\kappa(s)$ serves as $M$ in our proof of Conjecture
$\ref{Main}$ when $L \geq s+3$.

\begin{theorem}
\label{Large}
If $s$ and $L$ are positive integers with $L \geq s+3$ and $N \geq
\kappa(s)$, then 
\begin{equation*}
	|\{\pi \in C_{L,s} : |\pi| = N \}| > |\{\pi \in D_{L,s} : |\pi| = N\}|.
\end{equation*}
\end{theorem}

\begin{proof}
We construct an injective map 
\begin{equation*}
	\phi : \{\pi \in D_{L,s} : |\pi| = N\} \rightarrow \{\pi \in C_{L,s} : |\pi| = N \}.
\end{equation*} 

To show strict inequality, at the end of the proof we show that there is an element
in the codomain of $\phi$ not in its range. 

For $\pi \in D_{L,s}$, the image of $\pi$ under $\phi$ is given in cases 
depending on the frequency of $L+s-1$ in $\pi$.  Hence, for brevity, we set
$\freq=f_{L+s-1}$, so any $\pi \in D_{L,s}$ has the form
\begin{equation*}
	\pi = ((s+1)^{f_{s+1}}, \ldots, (L+s-1)^{\freq}, (L+s)^{f_{L+s}}). 
\end{equation*}
Our definition of the image of $\pi$ under $\phi$ is given in
two cases, when $\freq = 0$ and when $\freq \geq 1$, and
each case contains several subcases.  Our strategy for ensuring $\phi$ is injective is
to have images of partitions under $\phi$ from different cases have
different frequencies of $s$, while ensuring that in each case itself $\phi$ is injective.
To make our strategy and arguments on injectivity clear, we summarize in Table \ref{tab:tab1}  the
frequencies of $s$ in partitions in the image of $\phi$ for each
case.  As the right hand
column of Table \ref{tab:tab1} contains disjoint sets, the frequency of $s$ in a partition in the image
immediately determines from what case its preimage comes.  Then we only
need to ensure that $\phi$ is injective in each case.
\begin{table}[htpb]
    \centering
    \begin{tabular}{|c|c|}
 \hline
Case      &   Possible frequencies of $s$ \\
\hline
1(a)        &   Multiples of 12 \\
\hline
1(b)(i)     &   15 \\
\hline
1(b)(ii)    &   20 \\
\hline
1(b)(iii)   &   2,4,6,8 \\
\hline
2(a)        &   Odd numbers other than 15 \\
\hline
2(b)        &   14 \\
\hline
    \end{tabular}
\caption{The frequency of $s$ in the image of a partition under the function $\phi$ in the
    different cases for Theorem \ref{Large}.}
        \label{tab:tab1}
\end{table}

Case 1: Suppose that $\freq=0$ in $\pi$. In this case, we obtain
$\phi(\pi)$ by inserting some number of parts equal to $s$ into the partition
$\pi$; as $\freq = 0$, we do not need to remove the parts of size $L+s-1$,
but must compensate by altering the other parts of $\pi$. The
number of parts equal to $s$ to be inserted into $\pi$ is given by the subcases below.

Case 1(a): Suppose that there exists an $m$ such that $s+1 \leq m \leq
F(s)-1$ and $f_m \geq 12s$. Let $m_0$ be the least such number. Then
define 
\begin{equation*}
	\phi(\pi) = (s^{12m_0}, (s+1)^{f_{s+1}},\ldots m_0^{f_{m_0}-12s},\ldots).
\end{equation*}
We can see that $\phi$ is injective in this case because from the frequency of
$s$ in $\phi(\pi)$ we can easily determine $m_0$, and from this $\pi$ can be
recovered. 

Case 1(b): Suppose that the condition of Case 1(a) does not hold.  That
is, for every $m$ such that $s+1 \leq m \leq F(s)-1$, we have $f_m
\leq 12s - 1$. {\cred  Note that if such partitions do not exist, then Case
	$1$(b) does not arise, and
there is no need to construct an injection}. Since $N \geq \kappa(s)$, we must have $L+s \geq
F(s)$, and also there must exist an $h \geq F(s)$ such that $f_h >
0$. Let $l$ be the least such number. Thus, we can write $\pi$ as
$$ \pi = ((s+1)^{f_{s+1}},\ldots,(F(s)-1)^{f_{F(s)-1}}, \dots,
l^{f_{l}},\ldots).$$ 
We have some further subcases.

Case 1(b)(i):  If $f_{5s+1} \geq 1$ and $f_{10s-1} \geq 1$, then
define 
\begin{equation*}
	\phi(\pi) =
	(s^{15},(s+1)^{f_{s+1}},\ldots,(5s+1)^{f_{5s+1}-1},\ldots,(10s-1)^{f_{10s-1}-1},\ldots).
\end{equation*}
The injectivity of $\phi$ is clear in this case.

Case 1(b)(ii):  If $f_{5s+1} = 0$ or $f_{10s-1} = 0$ and  $f_{5s+2}
\geq 1$ and $f_{15s-2} \geq 1$, then define 
$$\phi(\pi) =
(s^{20},(s+1)^{f_{s+1}},\ldots,(5s+2)^{f_{5s+2}-1},\ldots,(15s-2)^{f_{15s-2}-1},\ldots).$$ 
The injectivity of $\phi$ is also clear in this case.

Case 1(b)(iii):  If $f_{5s+1} = 0$ or $f_{10s-1} = 0$ and $f_{5s+2}
= 0$ or $f_{15s-2} = 0$.  Then at least one of the following statements
is true:

\begin{itemize}
\item $T_1$: $f_{5s+1}=0$ and $f_{5s+2}=0$;
\item $T_2$: $f_{5s+1}=0$ and $f_{15s-2}=0$;
\item $T_3$: $f_{10s-1}=0$ and $f_{5s+2}=0$;
\item $T_4$: $f_{10s-1}=0$ and $f_{15s-2}=0$.
\end{itemize}

The indices in each of the statements are intentionally
chosen to be coprime with each other.
For example, let us show that $5s+1$ and $15s-2$ are coprime.  If
$g=\gcd(5s+1,15s-2)$, then $g \mid (5s+1)$ and $g \mid (15s-2)$. But
then $g \mid 3(5s+1)-(15s-2) = 5$. Therefore $g=1$ or $g=5$, but $g
\neq 5$ since $5 \nmid (5s+1)$.  The other pairs can be shown to be coprime with
similar ease.

Since $F(s)-8s = (10s-2)(15s-3)$ and the aforementioned indices in each
statement are coprime, by Lemma $\ref{Frobenius}$ the
following equations have nonnegative integer solutions for all $m \geq
F(s)$:

\begin{equation*}
    \begin{split}
        &\bullet\; (5s+1)x_m + (5s+2)y_m = m-2s;\\
        &\bullet\; (5s+1)z_m + (15s-2)w_m = m-4s;\\
        &\bullet\; (10s-1)u_m + (5s+2)v_m = m-6s;\\
        &\bullet\; (10s-1)p_m + (15s-2)q_m = m-8s.
    \end{split}\hspace{8.2cm}
\end{equation*}

That the lower bound on $m$ is sufficient for all the equations to
have nonnegative integer solutions follows from the lower bound being sufficient
for the last equation to have such solutions;  there the lower
bound on $m$ is the one specified by Lemma \ref{Frobenius}.  For each $m \geq F(s)$, fix some values of
$x_m$, $y_m$, $z_m$, $w_m$, $u_m$, $v_m$, $p_m$ and $q_m$ that satisfy
the equations, and keep these values fixed throughout the proof. Recall that
$l$ was defined to be the least number greater than or equal to
$F(s)$ that appears with nonzero frequency in the partition $\pi$.
Then we have the following cases:

\begin{itemize}
\item if $T_1$ is true, define 
    \begin{equation*}
        \phi(\pi) = (s^2, (s+1)^{f_{s+1}}, \ldots, (5s+1)^{x_{l}},
        (5s+2)^{y_{l}}, \ldots,(F(s)-1)^{f_{F(s)-1}}, \dots, l^{f_{l}-1},
        \ldots);
    \end{equation*}
\item if $T_1$ is false and $T_2$ is true, define 
    \begin{equation*}\phi(\pi) = (s^4,
	    (s+1)^{f_{s+1}}, \ldots, (5s+1)^{z_{l}}, \ldots, (15s-2)^{w_{l}},\\
        \ldots,(F(s)-1)^{f_{F(s)-1}}, \dots, l^{f_{l}-1}, \ldots);
    \end{equation*}
\item if $T_1$ and $T_2$ are false and $T_3$ is true, define 
    \begin{equation*} 
	    \phi(\pi) = (s^6, (s+1)^{f_{s+1}}, \ldots, (5s+2)^{v_{l}}, \ldots,
        (10s-1)^{u_{l}},\\ \ldots,(F(s)-1)^{f_{F(s)-1}}, \dots, l^{f_{l}-1},
        \ldots);
    \end{equation*}
\item if $T_1$, $T_2$ and $T_3$ are false and $T_4$ is true, define
    \begin{equation*}
	    \phi(\pi) = (s^8, (s+1)^{f_{s+1}}, \ldots, (10s-1)^{p_{l}}, \ldots,
        (15s-2)^{q_{l}},\\ \ldots,(F(s)-1)^{f_{F(s)-1}}, \dots, l^{f_{l}-1},
        \ldots).
    \end{equation*}

\end{itemize}

%
The function $\phi$ is injective in Case 1(b).  To see why, given a partition $\pihat =
\phi(\pi)$ whose frequency of $s$ is 2, 4, 6 or 8, we can recognize $\pi$ as
coming from this case.  Then if, for example, the frequency of $s$ in $\pihat$
is 2,  then $T_1$ is true, and the frequencies of $5s+1$ and $5s+2$ in $\pihat$ give the values of
$x_{l}$ and $y_{l}$, respectively.  Then from the defining equation for $x_l$ and
$y_l$ given by
\begin{equation*} 
    (5s+1)x_{l} + (5s+2)y_{l} = l-2s,
\end{equation*}
we can recover $l$.  From there it is easy to find $\pi$.  A similar argument
applies if the frequency of $s$ in $\pihat$ is 4, 6 or 8.  Thus, in all of Case
1(b), $\phi$ is injective.

This completes the description of $\phi$ for the case $\freq = 0$.   Note, in aggregate,
the function $\phi$ is injective in Case 1.  If $\pihat = \phi(\pi)$, then the frequency
of $s$ in $\pihat$ indicates from which subcase $\pi$ comes and, as shown above, $\pi$ is
then recoverable.

Case $2$: Suppose $\freq \geq 1$.  Hence, to produce the image of $\pi$
under $\phi$ in
this case, we must remove all
parts of size $L+s-1$.  Recall that $\pi$ has the form
\begin{equation*}
	\pi = ((s+1)^{f_{s+1}}, \ldots , (L+s-1)^\freq,(L+s)^{f_{L+s}}).
\end{equation*}
Since $L \geq s+3$, we have $(L-s-1)j \geq 1$ for all
$j \geq 1$, and therefore
\begin{equation} \label{First}
	(L+s-1)j-s(2j-1) \geq s+1.
\end{equation}
Then, by Lemma \ref{Simple}, for all $j \geq 1$, the equation 
\begin{equation}\label{eq:defr}
	(L+s-1)j = s(2j-1)+(s+1)r_{s+1,j} + (s+2)r_{s+2,j} + \cdots + (2s+1)r_{2s+1,j}
\end{equation}
has nonnegative integer solutions
$r_{s+1,j},r_{s+2,j},\ldots,r_{2s+1,j}$. For each $j \geq 1$, fix  a solution
$r_{s+1,j},r_{s+2,j},\ldots,r_{2s+1,j}$.

Case $2$(a): Suppose $\freq \neq 8$.  {\cmagenta Since $L \geq s+3$, we
have} $L+s-1 > 2s+1$.  Define 
\begin{equation}
    \begin{split}
	\phi(\pi) = (s^{2\freq-1}, (s+1)^{f_{s+1}+r_{s+1,\freq}},
		\ldots,(2s+1)^{f_{2s+1}+ r_{2s+1,\freq}},(2s+2)^{f_{2s+2}},\\
        \ldots,
        (L+s-2)^{f_{L+s-2}}, (L+s-1)^0, (L+s)^{f_{L+s}}).
\end{split}\label{eq:defcase2}
\end{equation}
The case $\freq=8$ is dealt with separately below to ensure injectivity of
$\phi$ since then $2\freq-1=15$, and the frequency of 15 for $s$ in partitions in
the image of $\phi$ has already been used in Case 1(b)(i).

Case $2$(b): Suppose $\freq = 8$. It follows from \eqref{First} that $$ 8(L+s-1)-15s \geq s+1,$$ and thus 
\begin{equation*}
	8(L+s-1)-14s \geq s+1.
\end{equation*}
Therefore, by Lemma $\ref{Simple}$, the equation
\begin{equation}\label{eq:defy}
    8(L+s-1)=14s+(s+1)t_{s+1} + (s+2)t_{s+2} + \cdots +
(2s+1)t_{2s+1}
\end{equation} has a nonnegative integer solution.
Fix a solution $t_{s+1}, t_{s+2}, \ldots, t_{2s+1}$ of this equation. Thus, for
\begin{equation*}
	\pi = ((s+1)^{f_{s+1}}, \ldots , (L+s-1)^8,(L+s)^{f_{L+s}}),
\end{equation*} 
we define 
\begin{multline*}
	\phi(\pi) = (s^{14}, (s+1)^{f_{s+1}+t_{s+1}}, \ldots,
	(2s+1)^{f_{2s+1}+t_{2s+1}}, (2s+2)^{f_{2s+2}},\\
	\ldots,
    (L+s-2)^{f_{L+s-2}}, (L+s-1)^0, (L+s)^{f_{L+s}}).
\end{multline*}
To see why $\phi$ is injective in Case 2, suppose that $\pihat = \phi(\pi)$ and
that the frequency of $s$ in $\pihat$ is either an odd number not equal to 15, or 14.  In the
former case, from the frequency of $s$, we can determine $f = f_{L+s-1}$ from
\eqref{eq:defcase2} for its preimage;  then from \eqref{eq:defr}, one can
determine the constants $r_{s+1, f}, \ldots, r_{2s+1, f}$.  From this, it is
clear the partition $\pi$ can be reconstructed, so $\phi$ is injective in Case 2(a).  In the latter case when the frequency
of $s$ is 14, we can apply a similar argument using \eqref{eq:defy}.  We
conclude $\phi$ is injective in Case 2.

 We refer the reader back to Table \ref{tab:tab1} to note that the map $\phi$ is injective overall. If
$\pihat = \phi(\pi)$, then the frequency of $s$ in $\pihat$ gives the case from
which $\pi$ came, and in each case itself $\phi$ was shown to be injective. 

The injection above shows that $$ |\{\pi \in C_{L,s} : |\pi| = N \}| \geq
|\{\pi \in D_{L,s} : |\pi| = N\}|,$$ for every $N \geq \kappa(s)$. To complete the
proof of Theorem \ref{Large}, we prove that the inequality is in fact strict. To
show this, we find a partition of $N$ that is in $C_{L,s}$ but not in the image
of $\phi$. Since $N \geq \kappa(s)$ is large enough, by Lemma \ref{Frobenius}, there
exist nonnegative integers $x_0$ and $y_0$ such that $$ N = 10s + (s+1)x_0 +
(s+2)y_0.$$ 
Consider the partition $\lambda_N = (s^{10}, (s+1)^{x_0},
(s+2)^{y_0})$ of $N$.  Since $L \geq 3$, we have $L+s-1 > s+2$, so  $\lambda_N$
is in $C_{L,s}$.  However, the partition $\lambda_N$ is not in the image of $\phi$ since the
frequency of $s$ in $\lambda_N$ is $10$, and 10 does not occur as a frequency of
$s$ in a partition in the image of $\phi$ by Table \ref{tab:tab1}.

\end{proof}

\subsection{Generalizing Theorem \ref{Large} if the impermissible part $k$ is large}

A careful analysis of the proof of Theorem \ref{Large} in the previous section shows that we have not used the
fact
that the impermissible part is one less than the largest allowable part.  Therefore, the
proof can be extended for a general impermissible part $k$ under some restrictions.  We
presented the proof for $k=L+s-1$ in the previous section first to keep the case analysis simpler and illustrate our
techniques for the later proofs.  In the proof below, we explain how the proof
of Theorem
\ref{Large} can be easily generalized.  

We modify the definitions of $F(s)$ and $\kappa(s)$. For $s \geq 1$, define the quantities:
\begin{align*}
    \begin{split}
        &\bullet\; F'(s) = (21s-2)(35s-3)+8s;\\
        &\bullet\; \kappa'(s) = (12s-1)((s+1)+(s+2)+\cdots (F'(s)-1))+1.
    \end{split}
    \hspace{3.5cm}
\end{align*}

\begin{theorem}
\label{Largek}
Suppose $L$, $s$ and $k$ are positive integers such that $L \geq s+2$ and $ 2s+2
\leq k \leq L+s$.  Then 
\begin{equation*}
	|\{\pi \in {\corange I_{L,s,k}} : |\pi| = N \}| > |\{\pi \in D_{L,s} : |\pi| = N\}|,
\end{equation*}
for any $N \geq
\kappa'(s)$.
\end{theorem}

\begin{proof}
The proof of Theorem \ref{Largek} is the same as the proof of Theorem \ref{Large} with
$L+s-1$ being replaced with $k$ everywhere, with some minor modifications.   Note that we do not replace $L+s$ with $k+1$;
we only change the impermissible part from $L+s-1$ to $k$.   Since $k \geq
2s+1$, when $L+s-1$ is replaced by $k$, the crucial equation \eqref{First}
becomes
 \begin{equation*} 
	kj-s(2j-1) \geq s+1
\end{equation*} 
and holds for all $j \geq 1$. The condition $k \geq 2s+2$ is required to ensure that the
impermissible part $k$ is different from $2s+1$, which may have been added as a part in the
analogue of \eqref{eq:defr}, given by
\begin{equation*}
	kj = s(2j-1)+(s+1)r_{s+1,j} + (s+2)r_{s+2,j} + \cdots + (2s+1)r_{2s+1,j}.
\end{equation*}

Finally, we observe that the proof of Theorem \ref{Large} requires modification if
the impermissible part $k$ is one of the numbers $5s+1$, $5s+2$, $10s-1$ or
$15s-2$ because, to produce the image of a partition under $\phi$, these
numbers were added as parts in Case 1(b).  If $k$ is one these numbers, then we can repeat the same
proof with the numbers $5s+1$, $5s+2$, $10s-1$ and $15s-2$ replaced with $7s+1$,
$7s+2$, $21s-1$ and $35s-2$, respectively, and it is this modification that
necessitates changing the constants $F(s)$ and $\kappa(s)$ to $F'(s)$ and $\kappa'(s)$,
respectively.  Note that for $s=1$, the numbers $10s-1$
and $7s+2$ coincide and are equal to $9$.  Thus, for $s=1$ and $k=9$, we choose a set of numbers different from
$5s+1$, $5s+2$, $10s-1$ and $15s-2$ ($6$, $7$, $9$ and $13$ when $s=1$). The choice of $7$, $13$, $21$ and $29$ (instead of $6$, $7$, $9$ and $13$)
solves the problem for $s=1$ and $k=9$.  
These are all the necessary modifications needed to obtain a proof of the theorem.
\end{proof}

\subsection{An analogue of Theorem \ref{Large} if the impermissible part $k$ is small}

Theorem \ref{Largek} requires the condition $k \geq 2s+2$. The next theorem
{\cred  focuses} on when $s+1 \leq k \leq 2s+1$. For $s \geq 1$,  we modify $F(s)$ and $\kappa(s)$ as follows:
\begin{align*}
    \begin{split}
        &\bullet\; F''(s) = (120s(s+1)-2)(180s(s+1)-3)+420s(s+1);\\
        &\bullet\; \kappa''(s) = (300s(s+1)-1)((s+1)+(s+2)+\cdots (F''(s)-1))+1.
    \end{split}
    \hspace{3.5cm}
\end{align*}

\begin{theorem}
\label{LargeL_Smallk}
Suppose $L$, $s$ and $k$ are positive integers such that $L \geq 3s+3$ and $s+1 \leq k \leq 2s+1$. Then
\begin{equation*}
	|\{\pi \in {\corange  I_{L,s,k}} : |\pi| = N \}| > |\{\pi \in D_{L,s} : |\pi| = N\}|,
\end{equation*}
for any $N \geq
\kappa''(s)$.
\end{theorem}

\begin{proof}
Although the proof of Theorem \ref{LargeL_Smallk} is similar in style and
essence to that of Theorem \ref{Large}, it requires several more substantial
modifications, so we explain them in detail.  

We again construct an injective map 
\begin{equation*}
	\phi : \{\pi \in D_{L,s} : |\pi| = N\} \rightarrow \{\pi \in {\corange
	I_{L,s,k}} : |\pi| = N \}.
\end{equation*} 
To show strict inequality, at the end of the proof we show that there is an element
in the codomain of $\phi$ not in its range.

For $\pi \in D_{L,s}$, the image of $\pi$ under $\phi$ is given in cases 
depending on the frequency of $k$ in $\pi$.  For brevity, we set
$\freq=f_{k}$, so any $\pi \in D_{L,s}$ has the form
\begin{equation*}
	\pi = ((s+1)^{f_{s+1}}, \ldots, k^{\freq}, \ldots, (L+s)^{f_{L+s}}). 
\end{equation*}
Our definition of the image of $\pi$ under $\phi$ is given in
two cases, when $\freq \geq 3$ and when $\freq \leq 2$, and
each case contains several subcases.  Our strategy for ensuring $\phi$ is
injective is, like in Theorem \ref{Large},
to have the image of partitions under $\phi$ from different cases have
different frequencies of $s$.

Case 1: Suppose $\freq \geq 3$.  Then there exists a unique $j(\freq) \geq 0$ such that 
\begin{equation*}
    (60(s+1)-3)j(\freq)+3 \leq \freq \leq (60(s+1)-3)(j(\freq)+1)+2.
\end{equation*}  
We regard $j(\cdot)$ as a function from $\{3, 4, ...\}$ to the set of
nonnegative integers.  Because $k \geq s+1$, for any $i \geq 3$, we have
\begin{equation*}
    k i - s(i+3j(i)-2) \geq 2s+2.
\end{equation*}
Then, by Lemma \ref{Simple} (applied with $s$ replaced by $2s+1$ there), for all $i \geq 3$, the equation 
\begin{equation}\label{Neweq:defr}
    k i = s(i+3j(i)-2)+(2s+2)r_{2s+2,i} + (2s+3)r_{2s+3,i} + \cdots + (4s+3)r_{4s+3,i}
\end{equation}
has nonnegative integer solutions
$r_{2s+2,i},r_{2s+3,i},\ldots,r_{4s+3,i}$. For each $i \geq 3$, fix  a solution
$r_{2s+2,i},r_{2s+3,i},\ldots,r_{4s+3,i}$.  Since $L \geq 3s+3$, we have $L+s
\geq 4s+3$, so $2s+2, \ldots, 4s+3$ are valid parts for partitions in $D_{L,s}$.  Define 
\begin{equation*}
    \phi(\pi) = (s^{\freq+3j(\freq)-2}, \ldots, k^0, \ldots,(2s+2)^{f_{2s+2}+ r_{2s+2,\freq}}, \ldots, (4s+3)^{f_{4s+3}+r_{4s+3,\freq}}, \ldots ).
\end{equation*}
Note that because $s+1 \leq k \leq 2s+1$, the frequency $k$ in $\phi(\pi)$ is genuinely 0,
as it is not one of the parts whose frequency has increased.

To see why $\phi$ is injective in Case 1, note that each member of the set 
\begin{equation*}
    V = \{i+3j(i)-2: i \geq 3\}
\end{equation*}
is uniquely determined by its defining value of $i$;  if $i > i'$,  then
$j(i) \geq j(i')$, and so $i + 3j(i) - 2 > i' + 3j(i') - 2$.  Thus, if $\pihat =
\phi(\pi)$ and the frequency of $s$ in $\pihat$ is in $V$, we can reverse the process above.  From
the frequency of $s$ in $\pihat$, we can recover $\freq$;  from there we can use
$\freq$, $j(\freq)$ and \eqref{Neweq:defr} to determine the constants
$r_{2s+2,\freq},r_{2s+3,\freq},\ldots,r_{4s+3,\freq}$.  From this point, determining
$\pi$ is straightforward.  We note here, for showing that $\phi$ is injective overall
later, that the members of $V$ are congruent to $1, 2, 3, 4, \ldots, - 3$ modulo
$60(s+1)$.  That is, no member of $V$ is congruent to $0,-1$ or $-2$ modulo
$60(s+1)$.

Case 2: Suppose $\freq \leq 2$.  As in Case 1, to obtain the image of a partition
under $\phi$, we must remove the parts of size $k$ (if any) and insert parts of $s$.  To
ensure $|\phi(\pi)| = |\pi|$, we must alter the frequencies of other parts to compensate.   The
number of parts equal to $s$ to be inserted into $\pi$ is given by the subcases
below.  We describe all subcases of Case 2, and then discuss why $\phi$ is
injective in Case 2.

Case 2(a): Suppose that there exists an $m$ such that $s+1 \leq m \leq
F''(s)-1$ and $f_m \geq 300s(s+1)$. Let $m_0$ be the least such number.  Notice that $m_0
\neq k$ because $f \leq 2$.  Then
define 
\begin{equation*}
	\phi(\pi) = (s^{300(s+1)m_0-\freq}, (s+1)^{f_{s+1}},\ldots, k^0, \ldots, (s+k)^{f_{s+k}+f}, \ldots, m_0^{f_{m_0}-300s(s+1)},\ldots).
\end{equation*}

Case 2(b): Suppose that the condition of Case 2(a) does not hold. That
is, for every $m$ such that $s+1 \leq m \leq F''(s)-1$, we have $f_m
\leq 300s(s+1) - 1$. {\cred Note that if such partitions do not exist, then Case
	$2$(b) does not arise, and
there is no need to construct an injection}. Since $N \geq \kappa''(s)$, we must have $L+s \geq
F''(s)$, and also there must exist an $h \geq F''(s)$ such that $f_h >
0$. Let $l$ be the least such number. Thus, we can write $\pi$ as
$$ \pi = ((s+1)^{f_{s+1}}, \ldots, k^\freq, \ldots,(F''(s)-1)^{f_{F''(s)-1}}, \dots,
l^{f_l},\ldots).$$ 
We have some further subcases. To ease notation, we define the following quantities:
\begin{equation*}
    \begin{split}
        &\bullet\; \alpha = 60s(s+1)+1;\\
        &\bullet\; \beta = 60s(s+1)+2;\\
        &\bullet\; \gamma = 120s(s+1)-1;\\
        &\bullet\; \delta = 180s(s+1)-2.
    \end{split}\hspace{11cm}
\end{equation*}
Note that the quantities $\alpha$, $\beta$, $\gamma$ and $\delta$ are chosen
such that they are distinct from $k$ and
\begin{equation*}
    \begin{split}
        &\bullet\; \alpha, \beta, \gamma, \delta < F''(s),\\
        &\bullet\; \gcd(\alpha, \beta) = \gcd(\alpha, \delta) = 1,\\
        &\bullet\; \gcd(\gamma, \beta) = \gcd(\gamma, \delta) = 1,\\
        &\bullet\; \alpha + \gamma = 180s(s+1),\\
        &\bullet\; \beta + \delta = 240s(s+1). 
    \end{split}\hspace{10cm}
\end{equation*}

Case 2(b)(i):  If $f_{\alpha} \geq 1$ and $f_{\gamma} \geq 1$, then
define 
\begin{equation*}
	\phi(\pi) =
	(s^{180(s+1)-\freq},(s+1)^{f_{s+1}}, \ldots, k^0, \ldots, (s+k)^{f_{s+k}+f}, \ldots, \alpha^{f_{\alpha}-1},\ldots, \gamma^{f_{\gamma}-1}, \ldots).
\end{equation*}

Case 2(b)(ii):  If $f_{\alpha} = 0$ or $f_{\gamma} = 0$ and  $f_{\beta}
\geq 1$ and $f_{\delta} \geq 1$, then define 
$$\phi(\pi) =
(s^{240(s+1)-\freq},(s+1)^{f_{s+1}}, \ldots, k^0, \ldots, (s+k)^{f_{s+k}+f}, \ldots, \beta^{f_{\beta}-1},\ldots, \delta^{f_{\delta}-1}, \ldots). $$

Case 2(b)(iii):  If $f_{\alpha} = 0$ or $f_{\gamma} = 0$ and $f_{\beta}
= 0$ or $f_{\delta} = 0$,  then at least one of the following statements
is true:

\begin{itemize}
\item $T_1$: $f_{\alpha}=0$ and $f_{\beta}=0$;
\item $T_2$: $f_{\alpha}=0$ and $f_{\delta}=0$;
\item $T_3$: $f_{\gamma}=0$ and $f_{\beta}=0$;
\item $T_4$: $f_{\gamma}=0$ and $f_{\delta}=0$.
\end{itemize}

Since $F''(s)-420s(s+1) = (\gamma-1)(\delta-1)$, and since the relevant numbers are
coprime, by Lemma $\ref{Frobenius}$ the
following equations have nonnegative integer solutions for all $m \geq
F''(s)$:

\begin{align}\label{eq:tscenarios}
    \begin{split}
        &\bullet\; \alpha x_m + \beta y_m = m-60s(s+1);\\
        &\bullet\; \alpha z_m + \delta w_m = m-120s(s+1);\\
	&\bullet\; \gamma u_m + \beta v_m = m-360s(s+1);\\
	&\bullet\; \gamma p_m + \delta q_m = m-420s(s+1).
    \end{split}
    \hspace{8.6cm}
\end{align}

That the lower bound on $m$ is sufficient for all the equations to
have nonnegative integer solutions follows from the lower bound being sufficient
for the last equation to have such solutions;  there the lower
bound on $m$ is the one specified by Lemma \ref{Frobenius}.  For each $m \geq F''(s)$, fix some values of
$x_m$, $y_m$, $z_m$, $w_m$, $u_m$, $v_m$, $p_m$ and $q_m$ that satisfy
the equations, and keep these values fixed throughout the proof. Recall that
$l$ was defined to be the least number greater than or equal to
$F''(s)$ that appears with nonzero frequency in the partition $\pi$.
Then we have the following cases:

\begin{itemize}
\item if $T_1$ is true, define 
    \begin{multline*}
        \phi(\pi) = (s^{60(s+1)-f}, (s+1)^{f_{s+1}}, \ldots, k^0, 
	\ldots, (s+k)^{f_{s+k}+f}, \ldots, \alpha^{x_{l}}, \beta^{y_{l}},\\
\ldots,(F''(s)-1)^{f_{F''(s)-1}},
\dots, l^{f_{l}-1}, \ldots); 
\end{multline*}
\item if $T_1$ is false and $T_2$ is true, define 
\begin{multline*}
        \phi(\pi) = (s^{120(s+1)-f}, (s+1)^{f_{s+1}}, \ldots, k^0, 
	\ldots, (s+k)^{f_{s+k}+f}, \ldots, \alpha^{z_{l}}, \ldots, \delta^{w_{l}},\\
\ldots,(F''(s)-1)^{f_{F''(s)-1}},
\dots, l^{f_{l}-1}, \ldots); 
\end{multline*}
\item if $T_1$ and $T_2$ are false and $T_3$ is true, define 
\begin{multline*}
 \phi(\pi) = (s^{360(s+1)-f}, (s+1)^{f_{s+1}}, \ldots, k^0, 
	\ldots, (s+k)^{f_{s+k}+f}, \ldots, \beta^{v_{l}}, \ldots, \gamma^{u_{l}},\\
\ldots,(F''(s)-1)^{f_{F''(s)-1}},
\dots, l^{f_{l}-1}, \ldots); 
\end{multline*}
\item if $T_1$, $T_2$ and $T_3$ are false and $T_4$ is true, define
 \begin{multline*}
        \phi(\pi) = (s^{420(s+1)-f}, (s+1)^{f_{s+1}}, \ldots, k^0, 
	\ldots, (s+k)^{f_{s+k}+f}, \ldots, \gamma^{p_{l}}, \ldots, \delta^{q_{l}},\\
\ldots,(F''(s)-1)^{f_{F''(s)-1}},
\dots, l^{f_{l}-1}, \ldots). 
\end{multline*}
\end{itemize}
Note in all cases $\phi(\pi)$ has at least one part of size $s$, no parts of
size $k$, and all parts are $\leq L+s$, so $\phi(\pi) \in I_{L,s,k}$.

Define the following sets:
\begin{itemize}
\item $V_1 = \{60(s+1)i: i \geq 1\}$;
\item $V_2 = \{60(s+1)i-1: i \geq 1\}$;
\item $V_3 = \{60(s+1)i-2: i \geq 1\}$.
\end{itemize}

Note that the frequency of $s$ in a partition in the image of $\phi$ in Case 2 lies in $V_1$, $V_2$ or
$V_3$ according to whether $\freq$ is $0$, $1$ or $2$, respectively.  To see why $\phi$ is
injective in Case 2, suppose that $\pihat = \phi(\pi)$ and that the frequency of $s$ in
$\pihat$ lies in one of those three sets.   This frequency can be a member of one of two
sets:
\begin{equation}\label{eq:possibles}
	\{300i(s+1) - h : i  \geq 2,\; h = 0,1,2\}\;\; \textnormal{ or }\;\; \{60i(s+1) - h :
	i=1,2,3,4,6,7,\;  h=0,1,2\}.
\end{equation}	
The former set of values pertains to Case 2(a), while the latter to Case 2(b).  These
possible values for the frequency of $s$ in $\pihat$ distinguish which subcase $\pi$ comes
from.  Say, for example, the frequency of $s$ in $\pihat$ is $120(s+1) - h$ for
some $h=0,1,2$.  From this,
we recover $f$ as $h$.  We also know from the frequency of $s$ in $\pihat$ that  
for $\pi$ the condition $T_1$ above is false, but
$T_2$ is true.  Hence, the frequencies of $\alpha$ and $\delta$ in $\pihat$ give
the values of $z_l$ and $w_l$, respectively.  We can then use the second
equation in \eqref{eq:tscenarios} to find the value of $l$.  Once the value of $l$ is known, 
we can reconstruct $\pi$.  When the frequency of
$s$ in $\pihat$ is some other value in the sets given in \eqref{eq:possibles}, we can
similarly reconstruct $\pi$.

Finally, we note that $\phi$ is injective overall.  As discussed in both Cases 1 and 2
separately, $\phi$ is injective.  However, the sets $V, V_1, V_2$ and $V_3$ are
all pairwise disjoint.  Indeed, members of $V_1, V_2$ and $V_3$ are congruent to 0, -1 and -2 modulo
$60(s+1)$, respectively, whereas members of $V$, as noted in Case 1, are not congruent to 0, -1 or -2
modulo $60(s+1)$.   As these are the possible values of the frequency of $s$ in a partition in the
image of $\phi$, they distinguish the different cases for preimages under
$\phi$, and we conclude that $\phi$ is injective.

The injection above shows that $$ |\{\pi \in {\corange I_{L,s,k}} : |\pi| = N \}| \geq
|\{\pi \in D_{L,s} : |\pi| = N\}|,$$ for every $N \geq \kappa''(s)$. To complete the
proof of Theorem \ref{Large}, we prove that the inequality is in fact strict by
finding a partition of $N$ that is in {\corange $I_{L,s,k}$} but not in the image
of $\phi$. Since $N \geq \kappa''(s)$ is large enough, by Lemma \ref{Frobenius} there
exist nonnegative integers $x_0$ and $y_0$ such that $$ N = 480s(s+1) + (2s+2)x_0 +
(2s+3)y_0.$$ This gives a partition $\lambda_N = (s^{480(s+1)}, (2s+2)^{x_0},
(2s+3)^{y_0})$ of $N$ that is in {\corange $I_{L,s,k}$} (because $s+1 \leq k \leq 2s+1$), but not in the image of $\phi$ since the
frequency of $s$ in $\lambda_N$ is $480(s+1)$, a number which is different from the
frequencies of $s$ in partitions in the image of $\phi$. 
\end{proof}

\section{Proofs of Theorems \ref{MostGen} and \ref{Genk}}
 \label{Generalize}

 In Section \ref{sec:smallk}, we prove $H_{L,s,k}(q)$ is eventually positive for
 $s+1 \leq k \leq L+s$ (Theorem \ref{General}).   The bound $M$ for which $N \geq M$ guarantees the
coefficient of $q^N$ in $H_{L,s,k}(q)$ is positive is initially given as depending on $L$
and $s$, so we do not immediately obtain a proof of Theorem \ref{MostGen}.  However, from this result
and results in Section \ref{Proofs}, we are able to obtain a quick proof of
Theorem \ref{MostGen} in Section \ref{sec:sim}.  As noted in Section
\ref{Intro}, Conjectures \ref{Main} and \ref{Similar} can be obtained as
corollaries of Theorem \ref{MostGen}.  In Section \ref{sec:bigk}, we prove
Theorem \ref{Genk}.

\subsection{The case $s+1 \leq k \leq L+s$ for $H_{L,s,k}(q)$} \label{sec:smallk}

For positive integers $L \geq 3$ and $s$, define:
\begin{align}\label{eq:defplsgammals}
    \begin{split}
        &\bullet\; P_{L,s} = (s+1)(s+2)\ldots(s+L);\\
        &\bullet\; \gamma(L,s) =
        \left(\left(s+1\right)+\left(s+2\right)+\cdots+\left(s+L\right)\right)\\
        &\hspace{6.2cm} \cdot \left(P_{L,s}^{\left(P_{L,s}^2-1\right)L+2}+\left(\left(P_{L,s}^2-1\right)L-2\right)P_{L,s}\right).
    \end{split}
\end{align}
The number $\gamma(L,s)$, which only depends on $L$ and $s$, serves as our
bound $M$ in the next theorem.

\begin{theorem}
\label{General}
For positive integers $L, s$ and $k$, with $L \geq 3$ and $s+1 \leq k \leq L+s$, the inequality
$$ |\{\pi \in {\corange I_{L,s,k}} : |\pi| =
N \}| > |\{\pi \in D_{L,s} : |\pi| = N\}| $$ holds for every $N \geq \gamma(L,s)$.
\end{theorem}

\begin{proof}
    For $N \geq \gamma(L,s)$, we construct an injective map 
\begin{equation*}
	\psi : \{\pi \in D_{L,s}
	: |\pi| = N \} \longrightarrow \{\pi \in {\corange I_{L,s,k}} : |\pi| = N \}.
\end{equation*}
To show strict inequality, at the end of the proof we show that there is an element
in the codomain of $\psi$ not in its range.

We shall separate our argument into cases and subcases depending on
the frequencies of parts of $\pi = ((s+1)^{f_{s+1}}, \ldots, k^{f_k}, \ldots,
(L+s)^{f_{L+s}})$ in the domain.  The
part whose frequency defines the cases is $k$;  hence, to simplify notation, we
set $\freq = f_k$, so a partition in the domain has the form
\begin{equation*}
  \pi = ((s+1)^{f_{s+1}}, \ldots, k^{f}, \ldots, (L+s)^{f_{L+s}}).
\end{equation*}

Case $1$: Suppose $\freq=0$.  Since $N \geq \gamma(L,s)$ is large enough, there is an $m$ such that
$s+1 \leq m \leq L+s$ and $f_m \geq s$. Let $m_0$ be the least such number.  Clearly, the
number $m_0$ cannot be $k$, since $\freq = 0$.  Then define
$$ \psi(\pi) = (s^{m_0}, (s+1)^{f_{s+1}},\ldots,m_0^{f_{m_0}-s},\ldots).$$  As
$f_k = f = 0$, and the frequency of $s > 0$ in $\psi(\pi)$, we see that $\psi(\pi) \in I_{L,s,k}$.

It is clear that $\psi$ is injective on the domain in this case, and the frequency of $s$
in a partition in its image is in  the set
\begin{equation*}
    U_1 = \{s+1, \dots, L+s\}.
\end{equation*}
Case $2$: Suppose that $ \freq \neq 0$ in $\pi$.  We have some subcases.  As the
partitions of {\corange $I_{L,s,k}$}
have no part equal to $k$ but must have a part of size $s$, to
obtain $\psi(\pi)$ from $\pi$ we remove the parts of size $k$, add parts of size
$s$, and compensate in some way so that $|\psi(\pi)| = |\pi|$.  Choose $\alpha$ and $\beta$ as follows:
\begin{itemize}
\item if $k \neq s+1$ and $k \neq s+2$, then choose $\alpha = 1$ and $\beta = 2$;
\item if $k=s+1$, choose $\alpha=2$ and $\beta =3$;
\item if $k=s+2$, choose $\alpha=1$ and $\beta=3$.
\end{itemize}

Since $L \geq 3$, we have $s+1 \leq s+\alpha < s+\beta \leq L+s$.  Importantly,
the numbers
$\alpha$ and $\beta$ are chosen so that $k \neq s+\alpha$ and $k \neq s+\beta$.
{\cpurple Furthermore, the numbers $s+\alpha$ and $s + \beta$ are coprime except possibly when
$k=s+2$;  in that case, if $s$ is odd, the greatest common divisor of $s+\alpha$
and $s+\beta$ is 2.  The case $k=s+2$ with $s$ is odd will require special treatment
below because of this.}

Case $2$(a): Suppose $\pi$ has $\freq \geq P_{L,s}^2$.  For any $j \geq
P_{L,s}^2$, let $h(j)$ be the integer satisfying $P^{h(j)}_{L,s} \leq j <
P^{h(j)+1}_{L,s}$, and consider the set
\begin{equation*}
    U_a = \{j + (h(j)-3)P_{L,s} -2 : j \geq P_{L,s}^2\}.
\end{equation*}
Clearly every $j \geq P_{L,s}^2$ gives a unique member of $U_a$.  Conversely,
from each member of $U_a$,
its defining value of $j$ can be recovered;  if $j > j'$, then $h(j)  \geq
h'(j)$, and so $j + (h(j)-3)P_{L,s} - 2 > j' + (h(j')-3)P_{L,s} - 2$.

As $N \geq \gamma(L,s)$, it is easy to verify, for any $j \geq P_{L,s}^2$, that
\begin{equation*}
	k j \geq s(j+(h(j)-3)P_{L,s} - 2)+(s+\alpha  -1)(s+\beta  -1).
\end{equation*}
Therefore, if $k \neq s+2$, or $k = s+2$ and $s$ is even, by Lemma \ref{Frobenius}, since $s+\alpha$
and $s+\beta$ are coprime, for any $j \geq P^2_{L,s}$, the equation
\begin{equation}\label{eq:defxj}
kj  = s(j+(h(j)-3)P_{L,s} -2)+ (s+\alpha)x_j + (s+\beta)y_j
\end{equation} 
has nonnegative integer solutions $x_j$ and $y_j$.  

{\cpurple We noted earlier that when $k=s+2$ and $s$ is odd then $s+\alpha = s+1$ and
$s+\beta =  s+3$
have a greatest common factor of 2, so we cannot immediately apply Lemma
\ref{Frobenius} to obtain $x_j$ and $y_j$ in \eqref{eq:defxj}.  We
can however overcome this difficulty as follows.  Since $P_{L,s}$
is even, the integer $ kj  - s(j+(h(j)-3)P_{L,s}-2)$ is also even for all $j$ and is
greater than or equal to $(s+1)(s+3)$.  Thus, it follows that $$ Q=\frac{kj  -
s(j+(h(j)-3)P_{L,s}-2)}{2}$$ is an integer greater than or equal to the
Frobenius number of the coprime integers $\frac{s+1}{2}$ and $\frac{s+3}{2}$.  Thus, by Lemma \ref{Frobenius}, the quantity $Q$ can be expressed as
a nonnegative integer combination of $\frac{s+1}{2}$ and $\frac{s+3}{2}$, which
can be used to find solutions $x_j$ and $y_j$ to \eqref{eq:defxj}.  So,
even in the case when $k=s+2$ and $s$ is odd, \eqref{eq:defxj} has nonnegative
integer solutions.}

In any case, for each $j \geq P_{L,s}^2$, fix a solution $x_j$ and $y_j$ to
\eqref{eq:defxj} and keep it fixed throughout the entire proof.

Define
\begin{equation*}
    \psi(\pi) = (s^{\freq+(h(f)-3)P_{L,s}-2}, \ldots,
(s+\alpha)^{f_{s+\alpha} + x_\freq}, \ldots,
(s+\beta)^{f_{s+\beta} + y_\freq},\ldots, k^0, \ldots), 
\end{equation*}
where it is understood that the part $k$ is not precisely placed (it may, for
example, be the case that $k < s+\beta$).

To see that $\psi$ is injective in this case, suppose that $\pihat =
\psi(\pi)$ and that the frequency of $s$ in $\pihat$ is in $U_a$.  As the defining values of members of $U_a$ can be
recovered, we can recover $f$ from the frequency of $s$.  Once $f$ is found, we
can use \eqref{eq:defxj} to find $x_f$ and $y_f$.  From here, we can easily
recover $\pi$.


We repeat the above strategy for the remaining cases. To that end,
define the set $$U_{b} =  \{P_{L,s}^h + (h-4)P_{L,s} : h \geq 3\} \cup
\{P_{L,s}^h + (h-4)P_{L,s} + 1 : h \geq 3\},$$
and note that the union is disjoint.  The reader is invited to confirm that  1)  each member of $U_b$
uniquely determines its defining value of $h$, and 2)  the sets
$U_1, U_a$
and $U_b$ are pairwise disjoint.   The image of $\pi$ under $\psi$ in each of the remaining
cases has its frequency of $s$ lie in
$U_b$.  Hence it suffices to show that when $\psi$ is restricted to the domain of the
remaining cases, it is injective.

The remaining case is when $f < P_{L,s}^2$.

Case $2$(b): Suppose that $0 < \freq < P_{L,s}^2$ in $\pi$.   Since $N \geq \gamma(s)$
is large enough, there exists an $h$ such that $1 \leq h \leq L$ and 
\begin{equation}\label{eq:fsh1}
  f_{s+h} \geq P_{L,s}^{(\freq-1)L+(h+2)}+((\freq-1)L+(h-2))P_{L,s}.
\end{equation} 
 For brevity of notation, for any $0 < i < P_{L,s}^2$ and $1 \leq h \leq L$, set
\begin{equation*}
    m_{i,h}= P_{L,s}^{(i-1)L+(h+2)}+((i-1)L+(h-2))P_{L,s}.
\end{equation*} 
We make a few key observations about the numbers $m_{i,h}$, the first of which
is that $m_{i,h} \in U_{b}$ for any valid $i$ and $h$.  Second, each number
$m_{i,h}$ is determined uniquely by its defining value of $(i,h)$.  To see why, if $m_{i,
h} = m_{i', h'}$, then, as the exponent in the first term of $m_{i,h}$ is at least three, we
have
\begin{equation*}
    (i-1)L + h+2 = (i'-1)L + h'+2,
\end{equation*}
and thus $h-h' = (i'-i)L$.
But since $1 \leq h, h' \leq L $, this implies that $h=h'$, and so $i'=i$.

Let $p$ be the least integer $1 \leq h \leq L$ for which \eqref{eq:fsh1} is
satisfied.  Notice that the restriction on $f$ prevents $s+p$ from being $k$.  By definition, 
\begin{equation}
    f_{s+p} \geq m_{\freq,p}.
  \label{eq:comp1}
\end{equation}
Notice that $m_{\freq, p}$ is divisible by $P_{L,s}$, and thus also by $(s+p)$;
hence, we can define $j_{\freq, p}$ by 
\begin{equation}\label{eq:jfk1}
  (s+p) j_{\freq,p} = sm_{\freq,p}.
\end{equation}  
From \eqref{eq:comp1} and \eqref{eq:jfk1}, it is easy to verify that 
\begin{equation*}
  f_{s+p} \geq j_{\freq,p} + 2s.
\end{equation*}
For any integer $u$, we define $\eta_u$ to be $1$ if $u$ is odd and $0$
otherwise.
One can easily verify that for any $1 \leq t \leq P_{L,s}^2$ and any $1 \leq i
\leq L$, we have
\begin{equation}\label{eq:fudge}
	2s(s+i)+tk -s \left(\delta_{k,s+2} \right) \eta_{s} \eta_{t}  \geq (s+\alpha-1) (s+\beta-1),
\end{equation}
{\cpurple We explain the peculiar term $s \left(\delta_{k,s+2}\right) \eta_{s} \eta_{t}$.  As before, when $k \neq s+2$, or $k = s+2$ and $s$ is even, the numbers $s+\alpha$ and $s+\beta$ are coprime,
so by Lemma \ref{Frobenius} there exist nonnegative integer solutions $z_{t,i}$ and $w_{t,i}$ such that 
\begin{equation}\label{eq:zti}
    2s(s+i)+tk = s \left(\delta_{k,s+2} \right) \eta_{s} \eta_{t} + (s+\alpha) z_{t,i} + (s+\beta) w_{t,i}.
\end{equation}
When $k=s+2$ and $s$ is odd, recall that $s+\alpha = s+1$ and $s+\beta = s+3$
are not coprime, but we can remedy this issue as before.   The term $s
\left(\delta_{k,s+2}\right) \eta_{s} \eta_{t}$ guarantees that the left hand side of \eqref{eq:fudge} is
even.   We then apply the same fix as earlier:  we divide the left
hand side of \eqref{eq:fudge} by 2, and the result is greater than the Frobenius
number of the coprime integers $\tfrac{s+1}{2}$ and $\tfrac{s+3}{2}$.  Then applying
Lemma \ref{Frobenius} gives us nonnegative integer solutions to \eqref{eq:zti}
in this case as well.}

In any case, for each $1 \leq t \leq P_{L,s}^2$ and $1 \leq i \leq L$, fix a solution
$z_{t,i}$ and $w_{t,i}$ to \eqref{eq:zti}.

Let $n_{\freq, p}$ be defined as $m_{\freq, p} + \left(\delta_{k,s+2} \right)
\eta_{s} \eta_{\freq}$.  Notice that $n_{\freq, p} \in U_b$.  To obtain $\psi(\pi)$ from $\pi$,  we add to $\pi$ a part $s$ with frequency 
$n_{\freq,p}$, reduce the frequency of $s+p$ by $j_{\freq,p}+2s$, remove the $f$
parts of $k$, and add $(s+\alpha)$ and $(s+\beta)$ with frequencies of
$z_{\freq,p}$ and $w_{\freq,p}$, respectively.  Thus, if $p \neq \alpha$ or $p \neq \beta$, we define 
\begin{equation*} 
    \psi(\pi) = \left(s^{n_{\freq,p}}, \ldots,
    (s+\alpha)^{f_{s+\alpha} + z_{\freq,p}}, \ldots, (s+\beta)^{f_{s+\beta} + w_{\freq,p}}, \ldots (s+p)^{f_{s+p} - j_{\freq,p} - 2s},
\ldots k^0, \ldots \right).
\end{equation*}
If $p = \alpha$, we define 
\begin{equation*} 
    \psi(\pi) = \left(s^{n_{\freq,\alpha}}, \ldots,
    (s+\alpha)^{f_{s+\alpha} + z_{\freq,\alpha}- j_{\freq,\alpha} - 2s}, \ldots, (s+\beta)^{f_{s+\beta} + w_{\freq,\alpha}},
\ldots k^0, \ldots \right).
\end{equation*}
If $p = \beta$, we define 
\begin{equation*} 
    \psi(\pi) = \left(s^{n_{\freq,\beta}}, \ldots,
    (s+\alpha)^{f_{s+\alpha} + z_{\freq,\beta}}, \ldots, (s+\beta)^{f_{s+\beta} + w_{\freq,\beta}-j_{\freq,\beta}-2s}, \ldots k^0, \ldots \right).
\end{equation*}
We first note that it follows from \eqref{eq:jfk1} and \eqref{eq:zti} that
$|\pi| = |\psi(\pi)|$.  As the frequency of $s$ in $\psi(\pi)$ is $n_{\freq,p}$,
it lies {\cmagenta in} $U_b$, as noted earlier.   To see why
$\psi$ is injective in this case, suppose that $\pihat = \psi(\pi)$ and that the frequency of
$s$ in $\pihat$ has the form $n_{\freq,p}$.   Recall that both $\freq$ and
$p$ can be recovered from $m_{\freq,p}$.  However, from $n_{\freq, p}$ it is
clear that we can recover $m_{\freq,p}$;  if $n_{\freq, p}$ is divisible by
$P_{L,s}$, then $m_{\freq,p} = n_{\freq,p}$, and otherwise $m_{\freq, p} =
n_{\freq,p} - 1$.  Thus $\freq$ and $p$ are recoverable from $n_{\freq,p}$ as
well.  From
there, using \eqref{eq:jfk1}, we can determine $j_{\freq,p}$.   Furthermore,
from $\freq$ and $p$, we can use \eqref{eq:zti} to determine $z_{\freq,p}$ and $w_{\freq,p}$.  From
this point, we can easily recover $\pi$.

We summarize the possible frequencies of $s$ in a partition in the image of $\psi$ and the
case to which it pertains in Table \ref{tab:tab2}.
\begin{table}
	\centering
	\begin{tabular}{|c|c|c|}
		\hline
		Case of $\pi$     & Value of $f$ &  Possible frequencies of $s$ in $\psi(\pi)$.\\
		\hline
		1 &  $f=0$ & $U_1$\\
		\hline
		2(a) & $f \geq P^2_{L,s}$ & $U_a$ \\
		\hline
		2(b) &  $0 < f < P^2_{L,s}$ & $U_b$ \\
		\hline
	\end{tabular}
	\caption{The possible frequencies of $s$ in partitions in the image of $\psi$.}
	\label{tab:tab2}
\end{table}
As discussed earlier, as $U_1, U_a$ and $U_b$ are pairwise disjoint, the
frequency of $s$ in a partition in the image of $\psi$ characterizes the case from which its preimage
comes.  The injectivity of $\psi$ over all cases follows from $\psi$ being injective in each case, which has been
demonstrated. 

The injection above shows that $$ |\{\pi \in {\corange I_{L,s,k}} : |\pi| = N \}| \geq |\{\pi \in D_{L,s} : |\pi| = N\}|,$$
for every $N \geq \gamma(L,s)$. To complete the proof of Theorem \ref{General},
we must prove that the inequality is in fact strict, so we find a partition
of $N$ that is in {\corange $I_{L,s,k}$} but not in the
image of $\psi$. {\cpurple Since $N \geq \gamma(L,s)$ is large enough, by Lemma
\ref{Frobenius} there
exist nonnegative integers $x_0$ and $y_0$ such that 
\begin{equation}\label{eq:outofrange}
	N = s \left( L+s+1 + \left(\delta_{k,s+2} \right) \eta_{s} \eta_{N-L}\right) + (s+\alpha)x_0 + (s+\beta)y_0.
\end{equation}
This gives a partition $$ \lambda_N = (s^{L+s+1 +
\left(\delta_{k,s+2} \right) \eta_{s} \eta_{N-L}}, (s+\alpha)^{x_0},
(s+\beta)^{y_0})$$ of $N$.  Again, the term $\left(\delta_{k,s+2} \right) \eta_s
\eta_{N-L}$ is to ensure that \eqref{eq:outofrange} has a solution even in the
case $k=s+2$ and $s$ is odd, as before.  It is easy to see $\lambda_N$ has the
desired properties.   Its frequency of $s$ is either $L+s+1$ or $L+s+2$ and its frequency of $k$ is 0, so
it is in $I_{L,s,k}$.  Furthermore, it is easy to see its frequency of $s$
is not a member of $U_1, U_a$ or
$U_b$, so $\lambda_N$ is not in the range of $\psi$}.
\end{proof}

\subsection{Proofs of Theorem \ref{MostGen} and Conjectures \ref{Main} and
\ref{Similar}}\label{sec:sim}

As noted earlier, while the bound in Theorem \ref{General} depends on both $L$
and $s$, we can use that theorem in tandem with  Theorems \ref{Largek} and
\ref{LargeL_Smallk} to give a proof of Theorem \ref{MostGen}.  For that, define
\begin{equation}
    \label{eq:defgammas}
    \Gamma(s) = \gamma(3s+2, s),
\end{equation}
where $\gamma(\cdot, \cdot)$ is defined as in \eqref{eq:defplsgammals}.

\begin{proof}[Proof of Theorem \ref{MostGen}]
    We show that if $N \geq \Gamma(s)$, the inequality
    \begin{equation*}
		|\{\pi \in {\corange  I_{L,s,k}} : |\pi| = N \}| > |\{\pi \in D_{L,s} : |\pi| = N\}|
    \end{equation*}
    holds.
    
    If $L \geq 3s+3$ and $2s+2 \leq k \leq L+s$, then by Theorem \ref{Largek}
    the inequality holds for all $N \geq \kappa'(s)$.  
    
    However, if $L \geq 3s+3$
    and $s+1 \leq k \leq 2s+1$, then by Theorem \ref{LargeL_Smallk} the
    inequality holds for all $N \geq \kappa''(s)$.   Thus, combining these two
    results, we see that for $L \geq 3s + 3$ and $s+1 \leq k \leq L+s$, if $N
    \geq \max(\kappa'(s), \kappa''(s))$ then the inequality holds.

    Finally, if $3 \leq L \leq 3s+2$ and $s+1 \leq k \leq L+s$, then by Theorem \ref{General}
	the inequality holds {\cred for all} $N \geq \gamma(3s+2,s) = \Gamma(s)$.
    It follows that for all $L \geq 3$ and $s+1 \leq k \leq L+s$, the inequality
    holds for all $N$ larger than the three constants $\kappa'(s), \kappa''(s)$
    and $\Gamma(s)$.  A simple comparison reveals that $\Gamma(s)$ is the
    largest of the three constants, and we give a short proof of this.  First, an easy
    calculation shows that $\kappa'(s) < 10^7 (s+1)^5$ and $\kappa''(s) < 10^{11}
    (s+1)^{10}$.  Next, we show that $\Gamma(s)$ is always larger than these numbers.
    Observe that $P_{3s+2,s} > (s+1)^{3s+2} \geq 32$ and thus $(P_{3s+2,s}^2-1) > 1000$.
    Using this, we find 
     \begin{align*}
     \Gamma(s) = \gamma(3s+2,s) &> P_{3s+2,s}^{P_{3s+2}^2-1} \\
     &> P_{3s+2,s}^{1000} \\
     &> (s+1)^{1000(3s+2)} \\
     &\geq (s+1)^{5000} \\
      &\geq 2^{4990}(s+1)^{10}  \\
      &> 10^{11} (s+1)^{10}   \\
      &> \max(\kappa'(s), \kappa''(s)), 
     \end{align*}
     completing the proof.
\end{proof}

\begin{remark}
If $s+2 \leq L \leq 3s+2$ and $k \geq 2s+2$, then we can use the bound $\kappa'(s)$ in
Theorem \ref{Largek} instead of the larger bound $\gamma(3s+2,s)$ suggested by the proof of Theorem \ref{MostGen}.
\end{remark}

As remarked in Section \ref{Intro}, setting $k=L+s-1$ and $k=L$ in Theorem
\ref{MostGen} proves Conjectures \ref{Main} and \ref{Similar}, respectively.  We state these
separately, however, in the next corollaries so that the bound $M$  for when the partition inequalities
hold is explicit. 

\begin{corollary}
    \label{thm:proofconj3}
    If $L \geq 3$ and $s$ are positive integers, then
\begin{equation*}
    |\{\pi \in C_{L,s} : |\pi| = N \}| \geq |\{\pi \in D_{L,s} : |\pi| =
	N\}|\;\;\; \textnormal{(Conjecture 1)}
\end{equation*}
and 
\begin{equation*}
    |\{\pi \in C^{*}_{L,s} : |\pi| = N \}| \geq |\{\pi \in D_{L,s} : |\pi| =
	N\}|\;\;\; \textnormal{(Conjecture 2)}
\end{equation*}
for all $N \geq \Gamma(s)$, where $\Gamma(s)$ is defined in
\eqref{eq:defgammas}.
\end{corollary}



\subsection{The proof of Theorem \ref{Genk}}\label{sec:bigk}

In the previous section, we proved Theorem \ref{MostGen}, and, as remarked in
Section \ref{Intro}, this proves Theorem \ref{Genk} in the cases $s+1 \leq k \leq
L+s$.   We will, however, be able to use Theorem \ref{MostGen} to prove Theorem
\ref{Genk} in general.  We first prove some lemmas.

%
%

\begin{lemma}
\label{Helpful}
	For positive integers $L \geq 3$, $s$ and $ k \geq s+1$, the difference $H_{L,s,k+s}(q) -  H_{L,s,k}(q)$ is nonnegative.
\end{lemma}

\begin{proof}
	The lemma is stated in the simplest form that we need.  We shall, however,
	prove a stronger statement:   for $L,s$ and $k$ as in the lemma, the difference $H_{L,s,k+i}(q) -  H_{L,s,k}(q)$ is nonnegative for all
	$s \leq i \leq L+s$.

	We have $$ H_{L,s,k+i}(q) - H_{L,s,k}(q) =
	\frac{q^{s+k}(1-q^i)}{(1-q^s)(1-q^{s+1}) \cdots (1-q^{L+s})}. $$ For $s \leq i
	\leq L+s$, the factor $(1-q^i)$ in the numerator is also present in the
	denominator, so it cancels.  Hence the difference $H_{L,s,k+i}(q) -
	H_{L,s,k}(q)$ is nonnegative.
\end{proof}

From Theorem \ref{MostGen} and Lemma \ref{Helpful}, it follows that $H_{L,s,k}(q)$
is eventually positive for all $k \geq s+1$ whenever $L \geq s$.  However, we are still
left with various cases when $L < s$. For example, if $s=10$ and $L=3$, then
Theorem \ref{MostGen} shows that
$H_{L,s,k}(q)$ is eventually positive whenever $k$ is $11, 12$ or $13$. Then,
according to Lemma \ref{Helpful}, the series $H_{L,s,k}(q)$ is eventually positive whenever
$k$ is $21, 22$ or $23$, which leaves the gap $14 \leq k \leq
20$.  To complete the proof of Theorem \ref{Genk},  a close analysis of the cases
covered by a combination of Theorem \ref{MostGen} and Lemma \ref{Helpful} gives
that, when
$L < s$, it suffices to prove that $H_{L,s,k}(q)$ is
eventually positive whenever $L+s < k \leq 2s$.  We do so in Lemma
\ref{thm:lemfinal}, but before that we need another lemma. 
\begin{lemma}
\label{Conjecture3}
	For positive integers $L \geq 3$ and $s$, the coefficient of $q^N$ in the series
	\begin{equation*}
	\frac{q^s-q^{L+s-1}}{(q^s;q)_{L+1}} 
	\end{equation*} 
    is positive whenever $N \geq \gamma(s,s)$, where $\gamma(\cdot, \cdot)$ is
    defined in \eqref{eq:defplsgammals}.
\end{lemma}
\begin{proof}
	Given a natural number $N$, 
\begin{itemize}
	\item $A_N$ is the set of partitions of $N$ with parts in $\{s,  \ldots,
			L+s\}$,  and $s$ appears as a part at least once;
	\item $B_N$ is the set of partitions of $N$ with parts in $\{s,  \ldots,
			L+s\}$,  and $L+s-1$ appears as a part at least once.
\end{itemize}
Proving the lemma is equivalent to showing that  $|A_N| \geq |B_N|$ for
all $N \geq \gamma(s,s)$. When $L \geq s+1$, this is easy to show for all $N
\geq 1$;  if $B_N$ is nonempty and $\pi \in B_N$, remove a
part of size $L+s-1$ from $\pi$ and insert  parts $s$ and $t$, where $t = L-1 \geq
s$.  This process clearly defines an injective function from $B_N$ to $A_N$.  We must therefore deal with the
case when $L \leq s$.  For that, given a natural number $N$, 
\begin{itemize}
	\item $C_N$ is the set of partitions of $N$ where the smallest
			part is $s$, all parts are $\leq L+s$, and $L+s-1$ does not
			appear as a part; 
	\item $D_N$ is the set of partitions of $N$
			with parts in the set $\{s+1,  \ldots,
			L+s\}$; 
	\item $E_N$ is the set of partitions of $N$ with parts in the set $\{s,  \ldots,
			L+s\}$, and $L+s-1$ does not
			appear as a part; 
	\item $F_N$ is the set of partitions of $N$ with parts in the set $\{s,  \ldots,
			L+s\}$.
\end{itemize}
Notice that $C_N = \{\pi \in C_{L, s} : |\pi| = N\}$ and $D_N = \{\pi \in D_{L,s}
	: |\pi| = N\}$. 
    We could use Corollary \ref{thm:proofconj3} here, but since $L \leq s$, we
    can use Theorem \ref{General} to obtain a stronger bound by setting $k =
    L+s-1$ there;  the inequality $|C_N| > |D_N|$ then holds for all $N \geq
    \gamma(s,s)$.   Since $C_N \subset E_N$, we also
have $|E_N| \geq |D_N|$ for all $N \geq \gamma(s,s)$. Notice that $B_N = F_N \setminus E_N$ and
$A_N = F_N \setminus D_N$. Hence $|A_N|\geq |B_N|$ for all $N \geq \gamma(s,s)$. 
\end{proof}
As noted above, the following result completes the proof of Theorem \ref{Genk}.

\begin{lemma}\label{thm:lemfinal}
	For positive integers $L$, $s$ and $k$ such that $3 \leq L < s$ and $L+s \leq k
	\leq 2s$, the coefficient of $q^N$ in $H_{L,s,k}(q)$ is positive whenever $N
    \geq \Gamma(s)$.
\end{lemma}

\begin{proof}
	Our proof is by strong induction. The base case ($k=L+s$) has
	already been proven in Theorem \ref{MostGen}.  Next, assume that for some $i$ such that
	$L+s \leq i < 2s$,  the coefficient of $q^N$ in $H_{L,s,j}(q)$ is positive whenever $N \geq \Gamma(s)$ for all $L+s \leq j \leq i$. We
	shall prove that the coefficient of $q^N$ in $H_{L,s,i+1}(q)$ is also positive whenever $N \geq \Gamma(s)$. Consider the difference
    \begin{align*} 
        H_{L,s,i+1}(q) - H_{L,s,i-L+2}(q) &= \frac{q^s(q^{i-L+2}-q^{i+1})}{(q^s;q)_{L+1}} \\ 
            &= \frac{q^{i-L+2}(q^s-q^{L+s-1})}{(q^s;q)_{L+1}}. 
    \end{align*} 
    Thus,  from Lemma \ref{Conjecture3}, the coefficient of $q^N$ in the difference $H_{L,s,i+1}(q) -
	H_{L,s,i-L+2}(q)$ is positive whenever $N \geq \gamma(s,s)+2s$ (because $i < 2s$).
	It is easy to verify, for any positive integer $s$, that $\Gamma(s) \geq \gamma(s,s)+2s$. Thus, 
	the coefficient of $q^N$ in $H_{L,s,i+1}(q) - H_{L,s,i-L+2}(q)$ is positive whenever $N \geq \Gamma(s)$.
	
	  The induction hypothesis states that the coefficient of $q^N$ in $H_{L,s,j}(q)$ is positive whenever $N \geq \Gamma(s)$ for all $L+s \leq j \leq
	i$. Combining this with Theorem \ref{MostGen}, the coefficient of $q^N$ in
    $H_{L,s,j}(q)$ is positive whenever $N \geq \Gamma(s)$  for all $s+1 \leq j
    \leq i$. Since $L \geq 3$ and $i \geq L+s$, we have $s+1 \leq
	i-L+2 \leq i$, and thus the coefficient of $q^N$ in $H_{L,s,i-L+2}(q)$ is positive whenever $N \geq \Gamma(s)$.
	
	Thus, we have shown whenever $N \geq \Gamma(s)$ that the coefficient of $q^N$ in both the series $H_{L,s,i+1}(q) - H_{L,s,i-L+2}(q)$ and $H_{L,s,i-L+2}(q)$ is positive. This shows that 
	the coefficient of $q^N$ in $H_{L,s,i+1}(q)$ is positive whenever $N \geq
	\Gamma(s)$, completing the induction argument. 
	\end{proof}

We collect all of these results together to complete the proof of Theorem \ref{Genk},
which, as noted in Section \ref{Intro}, generalizes Conjecture \ref{Analytic} and Theorem \ref{MostGen}.

\begin{proof}[Proof of Theorem \ref{Genk}]
	Suppose $N \geq \Gamma(s)$.  If $L \geq s$, then Theorem \ref{MostGen} and Lemma \ref{Helpful} immediately prove Theorem \ref{Genk}. 

If $L < s$ and $s+1 \leq k \leq L+s$, then Theorem \ref{General} completes the proof of
Theorem \ref{Genk}. If $L < s$ and $L+s < k \leq 2s$, then Lemma \ref{thm:lemfinal}
completes the proof of Theorem \ref{Genk}. Thus, combining these two cases, the theorem
holds for $L < s$ and $s+1 \leq k \leq 2s$. Since the result holds for $L < s$
and $s+1 \leq k \leq 2s$, an application of Lemma \ref{Helpful} covers the cases
$L < s$ and $k > 2s$.  This covers all cases and completes the proof.
\end{proof}

\section{Proof of Conjecture \ref{s=2}}
\label{GL2}

In this section, we prove Conjecture \ref{s=2}, which pertains to the series
$G_{L,2}(q)$.  In \cite{BerkovichAlexander2017SEPI},
Berkovich and Uncu found an alternative expression for $G_{L,2}(q)$.
\begin{theorem}[Berkovich and Uncu (2019)] For $L \geq 3$, 
\label{OtherGL2}
$$G_{L,2}(q) = \frac{H_{L,2,L}(q) }{1-q^L}.$$ 
\end{theorem}
From Theorem \ref{Genk}, we know that $H_{L,2,L}(q)$ is
eventually positive, as $H_{L,2,L}(q)$ is the particular case of $s=2$ and $k=L$ in that
    theorem, and from there it can be shown that  $G_{L,2}(q)$ is also
	eventually positive. However,  to prove Conjecture \ref{s=2}, we need to prove
	that the required series is not merely eventually positive, but that
    all its coefficients, with the exception of a few small terms, are nonnegative.
	We therefore need a method for the particular case $s=2$ in the series
	$H_{L,s,L}(q)$ that analyzes the coefficients of $q^n$ for small $n$.  As for the previous conjectures, our methods highly depend on Lemma \ref{Frobenius} and Lemma \ref{Simple}.

Recall the following notation from Section \ref{Intro}, which we require throughout this
section. For a positive integer $L \geq 3$, 

\begin{itemize}
    \item $C^*_{L,2} = I_{L,2,L}$ denotes the set of partitions such that the smallest part is $2$, all parts are $ \leq L+2$, and $L$ is not a part; 
\item $D_{L,2}$ denotes the set of nonempty partitions with parts
            in the set $\{3,4,\ldots,L+2\}$. 
\end{itemize}

From \eqref{eq:coeffHdef}, for $N \geq 0$, the coefficient of $q^N$ in $H_{L,2,L}(q)$ is
\begin{equation*}
	|\{\pi \in C^*_{L,2} : |\pi| = N\}| - |\{\pi \in D_{L,2} : |\pi| =N\}|.
\end{equation*}
We use this combinatorial interpretation along with Theorem \ref{OtherGL2} to obtain
information about $G_{L,2}(q)$.
We prove Conjecture \ref{s=2} by first proving it for large $L$ (i.e., $L \geq
11$), and then we prove it for smaller values (i.e., for $5 \leq L \leq 10$).
The cases $L=3$ and $L=4$ will be done separately afterwards.

\begin{theorem}\label{thm:glbig}
For $L \geq 11$, 
\begin{equation*}
    G_{L,2}(q)+q^3  \succeq 0.
\end{equation*}
\end{theorem}

\begin{proof}
For $N > 3$, we construct an injective map $$\phi : \{\pi \in D_{L,2} : |\pi| =
N\} \rightarrow \{\pi \in C^*_{L,2} : |\pi| = N \}. $$  For a partition $\pi$ 
in the domain, we let $f$ be the frequency of $L$, so $\pi$ has the form $(3^{f_3}, \ldots, L^{f}, \ldots, (L+2)^{f_{L+2}})$. 
Our definition of the image of $\pi$ under $\phi$ is given in
two cases, when $f > 0$ and $f = 0$, with
the latter case containing several subcases.   We describe all the cases first,
and show that $\phi$ is injective later.  For the reader wishing to look ahead,
Table \ref{tab:tab3} contains a summary of the cases.

Case $1$: Suppose $f > 0$. Since $L \geq 11$, for any $i  \geq 1$, we have $(L-8)i \geq
3$, and thus by Lemma \ref{Simple} (applied with $s=2$ there), there are nonnegative
integers $x_i$, $y_i$ and $z_i$ such that $$ Li = 8i+3x_i+4y_i+5z_i. $$ For each $i \geq
1$, fix some values of $x_i$, $y_i$ and $z_i$ and keep them fixed throughout the proof.
Define $$ \phi(\pi) = \left(2^{4f}, 3^{x_f+f_3}, 4^{y_f+f_4}, 5^{z_f+f_5},
6^{f_6}, \ldots, L^0, \ldots (L+2)^{f_{L+2}}\right).$$

Case 2: When $f=0$, to obtain $\phi(\pi)$ from $\pi$, there are no parts of $L$ to remove.
To obtain $\phi(\pi)$, we must insert parts of size 2 into $\pi$ and compensate in
some way.   For this, we must consider several subcases.  We denote the smallest part of $\pi$ by $s(\pi)$.

Case 2(A): When $s(\pi) \geq 5$, we define $$ \phi(\pi) = \left(2^1, (s(\pi)-2)^1,
(s(\pi))^{f_{s(\pi)}-1}, \ldots \right).$$

Case $2$(B): Suppose $s(\pi) \leq 4$, so $s(\pi)$ is either $3$ or $4$.

Case $2$(B)(i): If $f_4 \geq 1$, we define $$ \phi(\pi) = \left(2^2, 3^{f_3}, 4^{f_4-1},
\ldots \right).$$

Case $2$(B)(ii): Suppose $f_4 = 0$, so $s(\pi) = 3$. We have further subcases. 

Case $2$(B)(ii)(a): If $f_3 \geq 2$, we define $$ \phi(\pi) = \left(2^3, 3^{f_3-2},
5^{f_5}, \ldots \right).$$

Case $2$(B)(ii)(b): Suppose $f_3 = 1$. Then $\pi = \left(3, 5^{f_5}, \ldots \right)$.
Since $N > 3$, there exists an $m \geq 5$ such that $f_m \geq 1$. Let $m_0$ be the least
such number. We have further subcases depending on whether $m_0$ is odd or even.

Case $2$(B)(ii)(b)$(\alpha)$: If $m_0$ is odd, we define 
\begin{align*}
    \phi(\pi) &= \left(2^1, 3^0,
\ldots, \left(\frac{m_0+1}{2}\right)^2, \ldots, m_0^{f_{m_0}-1}, \ldots
\right) & \textnormal{ if } m_0 > 5,\\
\textnormal{ and }\; \phi(\pi) &= \left(2^1, 3^2, 4^0, 
5^{f_{5}-1}, \ldots \right) & \textnormal{ if } m_0 = 5.
\end{align*}
In both cases, the part 3 and a part $m_0$ were removed from $\pi$,
and a part 2 and two parts $\tfrac{m_0 + 1}{2}$ were inserted into
$\pi$ to obtain $\phi(\pi)$.  This ensures $|\phi(\pi)| = |\pi|$ regardless of
the value of $m_0$.

Case $2$(B)(ii)(b)$(\beta)$: If $m_0$ is even, we define 
\begin{align*}
    \phi(\pi) &= \left(2^1, 3^0, \ldots, \left(\frac{m_0}{2}\right)^1,
    \left(\frac{m_0}{2}+1\right)^1, \ldots, m_0^{f_{m_0}-1}, \ldots
    \right)& \textnormal{ if } m_0 > 6,\\
	\textnormal{ and }\; \phi(\pi) &= \left(2^1, 3, 4, 5^0, 6^{f_6-1}, \ldots, \right)& \textnormal{
	if } m_0 = 6.
\end{align*}
In both cases, the part 3 and a part $m_0$ were removed from $\pi$,
and a part 2, a part $\tfrac{m_0}{2}$ and a part
$\tfrac{m_0}{2}+1$ were inserted into
$\pi$ to obtain $\phi(\pi)$.  This ensures $|\phi(\pi)| = |\pi|$ regardless of
the value of $m_0$.

\begin{table}[htpb]
    \begin{threeparttable}
    \centering
    \begin{tabular}{|c|c|c|c|c|c|c|c|}
	 \hline
	 Case      & \multicolumn{5}{c|}{Description of case} &  $f_2$ in $\phi(\pi)$ &
	 Next parts\\
	\hline
	 1        & {\cred $f>0$} & & & & &   mult. of 4  & $*$\\
	\hline
	 2(A)     & $f=0$ & $s(\pi) \geq 5$ &  & & &  1 & $s(\pi) -2,
     s(\pi)$\tnote{$\dag$}\\
	\hline
	 2(B)(i)     & $f=0$ & $s(\pi) \leq 4$ & $f_4 \geq 1$ & & &  2 & $*$\\
	\hline
	 2(B)(ii)(a) & $f=0$ & $s(\pi) \leq 4$ & $f_4 = 0$ & $f_3 \geq 2$ & &  3 & $*$\\
	\hline
	 2(B)(ii)(b)($\alpha$) & $f=0$ & $s(\pi) \leq 4$ & $f_4 = 0$ & $f_3 = 1$ & $m_0$
	 odd & 1 & $\left(\tfrac{m_0 + 1}{2}\right)^2$\\
	\hline
	 2(B)(ii)(b)($\beta$) & $f=0$ & $s(\pi) \leq 4$ & $f_4 = 0$ & $f_3 = 1$ &  $m_0$
	 even & 1 & $\tfrac{m_0}{2}, \tfrac{m_0}{2}+1$\\
	\hline
 \end{tabular}
 \begin{tablenotes}
     \scriptsize
    \item[$\dag$]  $s(\pi)$ may appear with frequency 0. 
\end{tablenotes}
    \end{threeparttable}
	\caption{The frequency of $2$ in the image of a partition under the function $\phi$ in the
different cases for Theorem \ref{thm:glbig}.  The quantity $m_0$ is defined in Case
2(B)(ii)(b).  The column {\cmagenta ``Next parts"} indicate the second and third smallest
parts, which are equal in the Case 2(B)(ii)(b)($\alpha$).}
        \label{tab:tab3}
\end{table}

The map $\phi$ is easily seen to be injective in each case separately.  To see
it  is injective overall, observe that images under $\phi$ in most of the cases are
separated by the frequency of $2$;  the only cases in which images have the same
frequency of $2$ are Cases 2(A), 2(B)(ii)(b)($\alpha$) and 2(B)(ii)(b)($\beta$).  The Case
2(B)(ii)(b)($\alpha$) is separated from the other two cases by the frequency,  2, of the
second smallest part.  Finally, the Cases 2(A) and
2(B)(ii)(b)($\beta$) are separated by the difference between the second smallest
and the third smallest parts; this is $1$ for Case 2(B)(ii)(b)($\beta$) and at
least $2$ for Case 2(A).   These differences are listed in Table
\ref{tab:tab3}.

Thus we have shown that for $N > 3$,  
\begin{equation*}
	|\{\pi \in C^*_{L,2} : |\pi| = N\}| - |\{\pi \in D_{L,2} : |\pi| = N\}| \geq 0.
\end{equation*}
However, if $N=3$, this is not true;  it is easy to
see that  $|\{\pi \in D_{L,2} : |\pi| = 3\}| = 1$ and $|\{\pi \in C^*_{L,2} : |\pi| =
	3\}| = 0$.  Furthermore, if $N=1$ or $2$, we easily find that $|\{\pi \in
C^*_{L,2} : |\pi| = N\}| \geq |\{\pi \in D_{L,2} : |\pi| = N\}|$.  

Let $H_{L,2,L}(q) = \sum_{n \geq 0} a_n q^n$.   Then the above combinatorial results imply
that $a_n \geq 0$ for all $n \neq 3$, and $a_n = -1$ for $n = 3$, whence $H_{L,2,L}(q) +
q^3 \succeq 0$.  We can in fact make a stronger claim about the coefficients $a_n$
when $n \geq 14$;   we have $a_n \geq 1$ for all $n \geq 14$.  To see this, we
find a partition of $n$ in $C^*_{L,2}$ not in the image of
$\phi$. Recall $L \geq 11$.  For $n=14$, the partition $\pi_{14} = (2^1, 3^4)$ is in $C^*_{L,2}$ but not in
the image of $\phi$.  For $n \geq 15$,  we have $n - 11
\geq 4$, and thus, by Lemma \ref{Simple} (with $s=3$ there), there are nonnegative integers $x_n$, $y_n$, $z_n$ and $u_n$ such that
$$ n = 11 + 4x_n + 5y_n + 6z_n + 7u_n. $$ For each $n \geq 15$, fix some choice of $x_n$,
$y_n$, $z_n$ and $u_n$. Thus $\pi_n = (2^1, 3^3, 4^{x_n}, 5^{y_n}, 6^{z_n}, 7^{u_n})$ is a
partition of $n$ in $C^*_{L,2}$ not in the image of $\phi$.
 
Let $G_{L,2}(q) = \sum_{n \geq 0} b_n q^n$.  To prove the  theorem, we are required to
show $b_n \geq 0$ whenever $n \neq 3$, and $b_3 \geq -1$.  By Theorem \ref{OtherGL2}, for
any $n \geq 0$, 
\begin{equation*}
	b_n = a_n + a_{n-L} + a_{n-2L}+ \cdots.
\end{equation*}
Using the division algorithm, we have $n =Lq+r$ where $0 \leq r < L$, and thus we can rewrite $b_n$ as 
\begin{equation}
	\label{b_n}
  	b_n = \sum_{i=0}^{q} a_{Li+r}. 
\end{equation}
If $r \neq 3$, then none of the terms on the right hand side of \eqref{b_n} are negative and thus $b_n
\geq 0$ as required. If $r = 3$ and $q \geq 1$, then from \eqref{b_n}, we have $b_n \geq
a_{L+3} + a_3$. Since $L \geq 11$, we have $L+3 \geq 14$, and thus $a_{L+3} \geq 1$, which
implies $b_n \geq 0$. Finally, if $r=3$ and $q=0$, then $b_n = b_3 = a_3 = -1$. 
\end{proof}
 
Next, we prove Conjecture \ref{s=2} for $ 5 \leq L \leq 10$.  
 \begin{theorem}
 \label{SmallL}
 For $5 \leq L \leq 10$, $$ G_{L,2}(q)+q^3  \succeq 0.$$
 \end{theorem}
 
\begin{proof}
Let $N_L = \frac{L(L+3)}{2}+2$;  we give the values of $N_L$ in Table \ref{tab:tab4}.
For $5 \leq L \leq 10$ and $N \geq N_L$, we construct an injective map $$ \psi : \{\pi \in
D_{L,2} : |\pi| = N\} \rightarrow \{\pi \in C^*_{L,2} : |\pi| = N \}.$$

\begin{table}[htpb]
 \centering
\begin{tabular}{|c|c|c|c|c|c|c|}
\hline
$L$    &   $5$  &  $6$  &  $7$  &  $8$ &  $9$  & $10$ \\
\hline
$N_L$  &  $22$   &   $29$ &   $37$ &   $46$ &   $56$ &   $67$  \\
\hline
\end{tabular}
\caption{The table of values of $N_L$ versus $L$ for $5 \leq L \leq 10$.}
\label{tab:tab4}
\end{table}

We again let $f$ be the frequency of $L$ in $\pi \in D_{L,2}$; so, $\pi = (3^{f_3}, 4^{f_4}, \ldots, L^{f}
\ldots, (L+2)^{f_{L+2}})$. To define the image of $\pi$ under the map
$\psi$, we consider several cases depending on $f$.  We describe $\psi$ first
and explain why it is injective later.

Case $1$: Suppose that $f$ is a positive even number.  Then define $$ \psi(\pi) =
\left(2^{\frac{L f}{2}}, 3^{f_3}, \ldots, L^0, \ldots, (L+2)^{f_{L+2}}\right).$$

Case $2$: Suppose that $f$ is a positive odd number. Then define $$ \psi(\pi) =
\left(2^{L\left(\frac{f-1}{2}\right)+1}, 3^{f_3}, \ldots, (L-2)^{f_{L-2}+1}, \ldots, L^0,
\ldots, (L+2)^{f_{L+2}}\right).$$

Case $3$: Suppose $f = 0$. Since $N \geq N_L$ is large enough, there exists an $i$ such
that $3 \leq i \leq L+2$ and $f_i \geq 2$.   Let $i_0$ be the least such number. Note that
$i_0 \neq L$ since $f = 0$. We have further subcases depending on whether $i_0 = L+1$
or not. 

Case $3$(i): Suppose $i_0 \neq L+1$. Then define $$ \psi(\pi) = \left(2^{i_0},
3^{f_3}, \ldots, i_0^{f_{i_0}-2}, \ldots L^0, \ldots, (L+2)^{f_{L+2}}\right).$$

Case $3$(ii): Suppose $i_0 = L+1$. Then define $$  \psi(\pi) = \left(2^2, 3^{f_3},
\ldots, (L-1)^{f_{L-1}+2}, L^0, (L+1)^{f_{L+1}-2}, \ldots, (L+2)^{f_{L+2}}\right).$$

It is easy to see that $\psi$ is injective in each case.  To see that $\psi$ is injective overall, 
note that the frequency of $2$ modulo $L$ in the image distinguishes the cases with one
exception:  when $i_0 = L+1$ and $i_0 = L+2$.  The frequencies of 2 in the image in those
cases are $2$ and $L+2$, respectively.  However, while those frequencies are the same
modulo $L$, they are different numbers and so distinguish the cases. Hence the map $\psi$ is injective. 

Let $H_{L,2,L}(q) = \sum_{n \geq 0} a_{L,n}q^n$ and $G_{L,2}(q) = \sum_{n \geq 0}
b_{L,n}q^n$. Then, from Theorem \ref{OtherGL2}, for all $n \geq 0$, 
\begin{equation}\label{eq:relbnan}
	b_{L,n} = a_{L,n} + a_{L,n-L} + a_{L,n-2L}+\cdots.
\end{equation}
From the injectivity of $\psi$, we have $a_{L,N} \geq 0$ for all $5 \leq L \leq 10$ and $N
\geq N_L$. In fact, we can show $a_{L,N} \geq 1$ for all $5 \leq L \leq 10$ and $N \geq N_L$. To see
this, we find a partition in $C^*_{L,2}$ which is not in the image of
$\psi$.  For $L \geq 5$, note that $N_L \geq 2L+12$.  But from $N \geq N_L$, we conclude
$N - 2(L+3) \geq 6$.  Hence, by Lemma \ref{Frobenius}, there are nonnegative integers $x_L$ and $y_L$ such that $$ N = 2(L+3) +
3x_L + 4y_L. $$ For each $5 \leq L \leq 10$, fix some values of $x_L$ and $y_L$. Thus,
there is a partition $\lambda_L$ of $N$ given by $$ \lambda_L = (2^{L+3}, 3^{x_L},
4^{y_L}).$$ Note that $\lambda_L$ is not in the image of $\psi$ since the frequency of $2$
is $L+3$, which is not possible for any partition in the image of $\psi$.

We are therefore left with determining the nonnegativity of $a_{L,N}$ when $N \leq N_L$.  Since
$5 \leq L \leq 10$, the numbers
$a_{L,N}$ for $N \leq N_L$ can easily be calculated by, for example, a short Magma program.  The negative values of $a_{L,N}$ for $5 \leq L \leq 10$ and $N \leq
N_L$ are given in Table \ref{tab:tab5}.  In that table, we have also given the value of
$a_{L,N+L}$ when $a_{L,N}$ is negative.

\begin{table}
	\centering
\begin{tabular}{|c|c|c|c|}
\hline
$L$     &   $N$ & $a_{L,N}$ & $a_{L,N+L}$  \\
\hline
$5$      &  $3$ & $-1$ & $2$   \\
\hline
$5$     &   $7$ & $-1$ & $2$    \\
\hline
$6$    &   $3$ & $-1$ & $1$  \\
\hline
$7$    &   $3$ & $-1$ & $3$   \\
\hline
$7$    &   $9$ & $-1$ & $10$  \\
\hline
$8$  &   $3$  & $-1$  & $3$ \\
\hline
$9$  &   $3$  & $-1$ & $4$    \\
\hline
$10$  &   $3$  & $-1$ & $5$   \\
\hline
\end{tabular}
\caption{The values $5 \leq L \leq 10$ and $0 \leq N \leq N_L$ where $a_{L, N}$ is negative.
Also given is $a_{L,N+L}$ in those cases.}
\label{tab:tab5}
\end{table}

Using \eqref{eq:relbnan} and Table \ref{tab:tab5}, along with the facts that $a_{5,2} =1$ and
$a_{7,2} =1$, we conclude that $b_{L,N} \geq 0$ for $5 \leq L \leq 10$ if $N \neq 3$, and $b_{L,3} = -1$.
Hence, for all $5 \leq L \leq 10$, we see that $G_{L,2}(q)+q^3  \succeq 0$.
\end{proof}

Finally, we prove Conjecture \ref{s=2} for the cases $L=3$ and $L=4$ in the next
two theorems.

\begin{theorem}\label{thm:lequal3}
For $L=3$, $$G_{L,2}(q)+q^3 + q^9 + q^{15} \succeq 0.$$
\end{theorem}

\begin{proof}
For $N > 43 $, we construct an injective map $$ \phi : \{\pi \in D_{3,2} : |\pi| =
N\} \rightarrow \{\pi \in C^*_{3,2} : |\pi| = N \}.$$
Recall that each $\pi \in D_{3,2}$ has the form $\pi =
(3^{f_3},4^{f_4},5^{f_5})$, and each partition in $C^*_{3,2}$ has 2 as a part,
but cannot have 3 as a part.  	We have cases depending on the frequency of $3$ in $\pi$.
We fully describe $\phi$ and later show it is injective.

Case $1$: Suppose that $f_3$ is a positive even number.  Then define $$ \phi(\pi) =
\left(2^{\frac{3f_3}{2}}, 3^0, 4^{f_4}, 5^{f_5}\right).$$

Case $2$: Suppose that $f_3 \geq 3$ is an odd number. Then define $$ \phi(\pi) =
\left(2^{3\left(\frac{f_3-3}{2}\right)+2}, 3^0,  4^{f_4}, 5^{f_5+1}\right).$$

Case $3$: Suppose $f_3 = 1$.  Then $\pi = (3^1, 4^{f_4}, 5^{f_5})$. Since $N > 43
$, either $f_4 \geq 1$ or $f_5 \geq 1$. We have further subcases based on these conditions.

Case $3$(i): Suppose $f_4 \geq 1$. Then define $$ \phi(\pi) = \left(2^1, 3^0, 4^{f_4-1}, 5^{f_5+1}\right).$$

Case $3$(ii): Suppose $f_4 = 0$ and $f_5 \geq 1$. Then define $$ \phi(\pi) =
\left(2^4, 3^0,  5^{f_4-1}\right).$$

Case $4$: Suppose $f_3 = 0$. Then $\pi = (4^{f_4}, 5^{f_5})$. Since $N > 43 $,
either $f_4 \geq 8$ or $f_5 \geq 4$. We have further subcases based on these conditions.

Case $4$(i): If $f_4 \geq 8$, then define $$ \phi(\pi) = \left(2^{16}, 3^0,  4^{f_4-8}, 5^{f_5}\right).$$

Case $4$(ii): If $f_4 \leq 7$ and $f_5 \geq 4$, then define $$ \phi(\pi) =
\left(2^{10}, 3^0, 4^{f_4}, 5^{f_5-4}\right).$$

In each case it is clear that we can recover $\pi$ from $\phi(\pi)$.  To see $\phi$ is
injective overall, the frequency of $2$ in partitions in the image distinguishes the different cases;  in Cases
1 and 2, the frequency of 2 is congruent to 0 or 2 modulo 3, while in the other cases the
frequency of 2 is either 1, 4, 16 or 10.
Thus, the map $\phi$ is injective overall whenever $N > 43$. 

Let $H_{3,2,3}(q) = \sum_{n \geq 0} a_n q^n$ and $G_{3,2}(q) = \sum_{n \geq 0}
b_n q^n$.  From Theorem \ref{OtherGL2}, $$b_n = a_n + a_{n-3} + a_{n-6} +
\cdots.$$ The injectivity of $\phi$ shows that $a_n \geq 0$ for all $n > 43$.
For $n \leq 43$, $a_n$ can be calculated easily using a computer.  It can be verified that $a_n$ is negative only when $n$ is
either $3$, $5$, $9$, $13$ or $15$ and in all of these cases, $a_n = -1$.  Table
\ref{tab:tab6} contains the values of $a_n$ for $n \leq 18$. 
\begin{table}
    \centering
    \begin{tabular}{|c|c|c|c|c|c|c|c|c|c|c|c|c|c|c|c|c|c|c|c|}
        \hline
        $n$  &  $0$  &  $1$ &  $2$  &  $3$ &  $4$ &  $5$ &  $6$  &  $7$ &  $8$
             &  $9$ &  $10$ &  $11$ &  $12$ &  $13$ &  $14$ &  $15$ &  $16$ &
        $17$ &  $18$ \\
        \hline
        $a_n$  &  $0$  &  $0$ &  $1$  &  $-1$ &  $0$ &  $-1$ &  $1$  &  $0$ &
        $0$  &  $-1$ &  $1$ &  $0$ &  $1$ &  $-1$ &  $2$ &  $-1$ &  $2$ &  $0$ &
        $2$ \\
        \hline
    \end{tabular}
    \caption{The values $a_n$ for $0 \leq n \leq 18$ in Theorem
    \ref{thm:lequal3}.}
    \label{tab:tab6}
\end{table}
From Table \ref{tab:tab6} and the fact that $a_n$ is negative only when $n$ is either
$3$, $5$, $9$, $13$ or $15$ (and in all of these cases, $a_n = -1$), we obtain
$b_n \geq 0$ whenever $n \neq 3$, $n \neq 9$ and $n \neq 15$. Also apparent is
that $b_3 = -1$, $b_9 =-1$ and $b_{15} =-1$.  This proves that $G_{3,2}(q)+q^3 + q^9 + q^{15} \succeq 0$.
\end{proof}

\begin{theorem}
    \label{thm:lequal4}
 For $L=4$, $$G_{L,2}(q)+q^3 + q^9 \succeq 0.$$
\end{theorem}

\begin{proof}
For $N > 20 $, we construct an injective map $$ \psi : \{\pi \in D_{4,2} : |\pi| =
N\} \rightarrow \{\pi \in C^*_{4,2} : |\pi| = N \}.$$
Recall that partitions $\pi \in D_{4,2}$ have the form $\pi = (3^{f_3},
4^{f_4}, 5^{f_5}, 6^{f_6})$, whereas partitions in $C^*_{4,2}$ must have a
part of size 2 and cannot have a part of size 4.  We have cases depending on the frequency
of $4$ in $\pi$.  We describe the map $\psi$ fully and then show it is injective later.

Case $1$: Suppose $f_4 \geq 1$. Then define $$ \psi(\pi) = (2^{2f_4},
3^{f_3}, 4^0, 5^{f_5}, 6^{f_6}).$$

Case $2$: Suppose $f_4 = 0$. Then $\pi = (3^{f_3}, 5^{f_5}, 6^{f_6})$.  Since $N > 20 $, either $f_3 \geq 2$, $f_5 \geq 2$ or $f_6 \geq 3$. We consider subcases based on these conditions. 

Case $2$(i): Suppose $f_3 \geq 2$. Then define $$ \psi(\pi) = (2^3,
3^{f_3-2}, 4^0, 5^{f_5}, 6^{f_6}).$$

Case $2$(ii): Suppose $f_3 \leq 1$ and $f_5 \geq 2$. Then define $$ \psi(\pi)
= (2^5, 3^{f_3}, 4^0, 5^{f_5-2}, 6^{f_6}).$$

Case $2$(iii): Suppose $f_3 \leq 1$, $f_5 \leq 1$ and $f_6 \geq 3$. Then 
define $$ \psi(\pi) = (2^9, 3^{f_3}, 4^0, 5^{f_5}, 6^{f_6-3}).$$

It is easy to see that $\psi$ is injective in each case.  To see that $\psi$ is injective
overall, notice that the frequency of $2$ in partitions in the image distinguishes cases.  In
Case 1, the frequency of 2 is even, and in the other cases the frequency of 2 is 3, 5 or
9.  Thus, $\psi$ is injective.

Let $H_{4,2,4}(q) = \sum_{n \geq 0} a_n q^n$ and $G_{4,2}(q) = \sum_{n \geq 0}
b_n q^n$. Then $$b_n = a_n + a_{n-4} + a_{n-8} + \cdots.$$ The injectivity of
$\psi$ shows that $a_n \geq 0$ for all $n > 20$. For $n \leq 20$, $a_n$ can be
calculated easily using a computer (or by hand).  It can be 
verified that $a_n$ is negative only when $n$ is either $3$, $6$ or $9$, and in
all of these cases, $a_n = -1$. Table \ref{tab:tab7} contains the values of $a_n$ for $n \leq
13$. 
\begin{table}
    \centering
    \begin{tabular}{|c|c|c|c|c|c|c|c|c|c|c|c|c|c|c|c|c|c|c|c|}
        \hline
        $n$  &  $0$  &  $1$ &  $2$  &  $3$ &  $4$ &  $5$ &  $6$  &  $7$ &  $8$
             &  $9$ &  $10$ &  $11$ &  $12$ &  $13$  \\
        \hline
        $a_n$  &  $0$  &  $0$ &  $1$  &  $-1$ &  $0$ &  $0$ &  $-1$  &  $1$ &
        $1$  &  $-1$ &  $1$ &  $1$ &  $0$ &  $2$  \\
        \hline
    \end{tabular}
    \caption{The values $a_n$ in Theorem \ref{thm:lequal4} for $0 \leq n \leq
    13$.}
    \label{tab:tab7}
\end{table}

From Table \ref{tab:tab7} and the fact that $a_n$ is negative only when $n$ is either $3$, $6$ or $9$ (and in all of these cases, $a_n = -1$), we get $b_n \geq 0$ whenever $n \neq 3$ and $n \neq 9$. Also, $b_3 = -1$ and $b_9 =-1$. This proves that $G_{4,2}(q)+q^3 + q^9 \succeq 0$.

\end{proof}

\section*{Acknowledgements}

We thank the reviewers for their encouraging comments and their many helpful
suggestions.

\bibliographystyle{halpha}
\bibliography{Integer}

\end{document}